\documentclass{article}
\usepackage[top=1.2in, bottom=1.2in, left=1in, right=1in]{geometry}
\usepackage{amsmath,amssymb,amsfonts,amsthm,xcolor}
 \pdfoutput=1

\usepackage{ifpdf}
\usepackage{hyperref}
\usepackage{booktabs}
\usepackage{ifthen}
\usepackage{xspace}
\usepackage[pdftex]{graphicx}
\usepackage{multirow}
\usepackage{mathrsfs}
\usepackage{enumerate}
\usepackage{enumitem}
\usepackage{subfigure}
\usepackage{bm}

\usepackage{fancyhdr}
\fancypagestyle{plain}{ %
  \fancyhf{} % remove everything
}
\makeatletter
\newcommand{\@KEYWORDS}{\@empty}
\newcommand{\KEYWORDS}[1]{\renewcommand{\@KEYWORDS}{#1}}
\newcommand{\makekeywords}{\begin{flushleft}{\footnotesize \textbf{\textit{Keywords ---}} \textit{\@KEYWORDS}}\end{flushleft}}
\makeatother

\usepackage[font=small]{caption}
\newcounter{dummycounter}
\newcommand{\ii}[1]{\setcounter{dummycounter}{#1}(\roman{dummycounter})\xspace}

\newcommand{\R}{\mathbb R}

\newcommand{\N}{\mathbb N}

\newcommand{\set}[2][]{\ifthenelse{\equal{#1}{\empty}}{\lbrace #2 \rbrace}{\lbrace #2\vert\, #1\rbrace}}
\newcommand{\diag}{\operatorname{diag}}

\newcommand{\tensor}{\ensuremath{\otimes}}

\makeatletter
\def\INT{
  \@ifnextchar[
    {\INT@i}
    {\INT@i[]}
}
\def\INT@i[#1]{
  \@ifnextchar[
    {\INT@ii{#1}}
    {\INT@ii{#1}[]}
}
\def\INT@ii#1[#2]#3#4{
\ensuremath{\int_{#1}^{#2}{#3}\,\mathit d{#4}}
}
\makeatother

\newcommand{\seq}[2][]{(#2)_{#1}}
\newcommand{\rpartial}{\widehat\partial}

\newcommand{\ind}[1]{\delta_{#1}}
\newcommand{\eR}{\overline{\R}}
\newcommand{\Graph}{\operatorname{Graph}}% \Gamma(#1)}%\operatorname{Graph}}
\newcommand{\dom}{\operatorname{dom}}
\newcommand{\epi}{\operatorname{epi}}

\newcommand{\abs}[1]{\vert #1 \vert}
\newcommand{\norm}[2][]{\Vert #2 \Vert_{#1}}

\newcommand{\scal}[2]{\left\langle #1,#2 \right\rangle}
\newcommand{\map}[3]{#1\colon #2 \to #3}
\newcommand{\smap}[3]{#1\colon #2 \rightrightarrows #3}
\newcommand{\fto}{\overset{f}{\to}}
\newcommand{\Fto}[1][]{\overset{#1}{\to}}

\newcommand{\argmin}[1]{\ensuremath{\underset{#1}{\operatorname{argmin}}\,}}

\newcommand{\st}{s.t.\xspace}

\newcommand{\dist}{\operatorname{dist}}

\newcommand{\spSob}[2][]{\ensuremath{H^{#2\ifthenelse{\equal{#1}{\empty}}{}{,#1}}}}
\usepackage{mathrsfs}
\newcommand{\cd}[1]{\ensuremath{C^{#1}}}

\renewcommand{\phi}{\varphi}
\newcommand{\eps}{\varepsilon}
\newcommand{\enquote}[1]{``#1''}

\renewcommand{\o}{\mathrm{o}}
\newcommand{\structO}{\mathcal O}

\newcommand{\opid}{\mathrm{id}}      % identity operator

\definecolor{NoteColor}{rgb}{1,0.4,0}
\definecolor{SkyRed}{rgb}{.56,.06,.15}
\definecolor{dSkyBlue}{rgb}{.11,.2,.62}
\definecolor{bSkyBlue}{rgb}{.13,.73,.92}

\newcommand{\leadingzero}[1]{\ifnum #1<10 0\the#1\else\the#1\fi}
\newcommand{\todayIII}{\leadingzero{\day}/\leadingzero{\month}/\the\year}

\usepackage[capitalise]{cleveref}
\newcounter{THMCTR}
\newtheorem{THM}[THMCTR]{Theorem}
\newtheorem{PROP}[THMCTR]{Proposition}
\newtheorem{DEF}[THMCTR]{Definition}

\newtheorem{LEM}[THMCTR]{Lemma}
\newtheorem{COR}[THMCTR]{Corollary}

\newtheorem{ASS}{Assumption}

\newtheorem{ALG}{\textbf{\textup{Algorithm}}}
\theoremstyle{remark}
\newcounter{EXCTR}
\newtheorem{REM}[EXCTR]{Remark}

\crefname{THM}{Theorem}{Theorems}
\crefname{PROP}{Proposition}{Propositions}
\crefname{DEF}{Definition}{Definitions}
\crefname{DEFPROP}{DECIDE DEF OR PROP}{DECIDE DEF OR PROP}
\crefname{LEM}{Lemma}{Lemmas}
\crefname{COR}{Corollary}{Corollaries}
\crefname{REM}{Remark}{Remarks}
\crefname{PARA}{Paraphrase}{Paraphrases}
\crefname{ASS}{Assumption}{Assumptions}
\crefname{EX}{Example}{Examples}
\crefname{ALG}{Algorithm}{Algorithms}

\crefname{part}{Part}{Parts}
\crefname{chapter}{Chapter}{Chapters}
\crefname{section}{Section}{Sections}
\crefname{subsection}{Section}{Sections}
\crefname{subsubsection}{Section}{Sections}
\crefname{figure}{Figure}{Figures}
\crefname{table}{Table}{Tables}

\newcommand{\MARK}[2][]{%
\color{SkyRed}#2\ifthenelse{\equal{#1}{\empty}}{}{%
  \marginpar{\textcolor{dSkyBlue}{\ensuremath{\clubsuit} \large \textsc TODO}\\\color{dSkyBlue}\footnotesize#1\hspace*{\fill}\\\color{black}}%
  \textcolor{SkyRed}{\ensuremath{{}^\clubsuit}}%
  }\color{black}\xspace%
}

\newcommand{\Z}{\mathbb Z}
\newcommand{\n}{{n}}
\newcommand{\np}{{n+1}}
\newcommand{\nm}{{n-1}}
\renewcommand{\k}{{k}}
\newcommand{\kp}{{k+1}}

\newcommand{\X}{X}
\newcommand{\img}{w}
\newcommand{\mask}{c}
\newcommand{\edge}{z}
\newcommand{\noisy}{I}
\newcommand{\dimX}{N}
\newcommand{\spd}{\mathbb S_{++}}       % symmetric positive definite matrix
\newcommand{\pit}[1]{_{#1}}             % iteration index for scalars
\newcommand{\iter}[1]{^{#1}}            % iteration index for vectors
\newcommand{\setsep}{\,\vert\,}

\definecolor{miared}{rgb}{0.7,0.15,0}
\definecolor{miablue}{rgb}{0.058,0.047,.596} 
\definecolor{mygreen}{rgb}{0.058,.596,0.047}

\newcommand{\F}{\mathcal F}
\newcommand{\crit}{\operatorname{crit}}

\newcommand{\dimN}{N}
\newcommand{\dimJ}{J}
\newcommand{\dimM}{M}
\newcommand{\dimP}{P}
\newcommand{\bx}{\mathbf{x}}
\newcommand{\by}{\mathbf{y}}
\newcommand{\bA}{\mathbf{A}}
\newcommand{\bH}{H}
\newcommand{\bdelta}{\mathbf{\Delta}}
\newcommand{\ov}[1]{\overline{#1}}

\newcommand{\BCVMiPiano}{\texttt{BC-VM-iPiano}\xspace}
\newcommand{\VMiPiano}{\texttt{VM-iPiano}\xspace}
\newcommand{\BCiPiano}{\texttt{BC-iPiano}\xspace}
\newcommand{\iPiano}{\texttt{iPiano}\xspace}
\newcommand{\FB}{\texttt{FB}\xspace}
\newcommand{\BCFB}{\texttt{BC-FB}\xspace}
\newcommand{\VMFB}{\texttt{BC-FB}\xspace}
\newcommand{\BCVMFB}{\texttt{BC-VM-FB}\xspace}

\newcommand{\KL}{Kurdyka--{\L}ojasiewicz\xspace}
\newcommand{\Loj}{{\L}ojasiewicz\xspace}

\usepackage{tikz}

\graphicspath{{./figures/}}

\title{Unifying abstract inexact convergence theorems\\ 
       and
       block coordinate variable metric iPiano}

\author{Peter Ochs\\
  Mathematical Optimization Group\\
  Saarland University\\
  Germany\\
  {\tt\small ochs@math.uni-sb.de}
}
\date{\today}


\begin{document}

\maketitle

% ********************
% >>>>> ABSTRACT <<<<<
% ********************
\begin{abstract}
An abstract convergence theorem for a class of generalized descent methods that
explicitly models relative errors is proved. The convergence theorem
generalizes and unifies several recent abstract convergence theorems. It is
applicable to possibly non-smooth and non-convex lower semi-continuous
functions that satisfy the \KL (KL) inequality, which comprises a huge class of
problems. Most of the recent algorithms that explicitly prove convergence using
the KL inequality can cast into the abstract framework in this paper and, 
therefore, the generated sequence converges to a stationary point of the 
objective function.
Additional flexibility compared to related approaches is gained by a descent
property that is formulated with respect to a function that is allowed to
change along the iterations, a generic distance measure, and an
explicit/implicit relative error condition with respect to finite linear 
combinations of distance terms.

As an application of the gained flexibility, the convergence of a block
coordinate variable metric version of iPiano (an inertial forward--backward
splitting algorithm) is proved, which performs favorably on an inpainting
problem with a Mumford--Shah-like regularization from image processing. 
\end{abstract}

\makekeywords

%\tableofcontents

%%%%%%%%%%%%%
%% SECTION %%
%%%%%%%%%%%%%
\section{Introduction}

The \KL (KL) inequality is key for the convergence analysis for non-smooth and
non-convex optimization problems. {\L}ojasiewicz introduced an early version of
this inequality for analytic functions \cite{Loj63}, which was extended to more
general classes of smooth functions in \cite{Kurd98,Loj93,KP94} and to
non-smooth functions (that are definable in an $\o$-minimal structure
\cite{Dries98}) in \cite{BDL06,BDLS07}. While it was originally used to study
the asymptotic behavior of gradient-like systems
\cite{BDL06,HJ98,HT01,Lageman07} and PDEs \cite{CJ03,Simon83}, the KL
inequality is also used for numerical methods such as the gradient method
\cite{AMA05}, proximal methods \cite{AB09}, projection or alternating
minimization methods \cite{ABRS10,BCP10}. A unifying and concise formulation of
the key ingredients, which, combined with the KL inequality, lead to asymptotic
convergence to a critical point and a trajectory with finite length (the
accumulated distance between consecutive points of the sequence is finite) is
proposed by Attouch et al. \cite{ABS13} and further refined by Bolte et al.
\cite{BST14} using a uniformization result for the KL inequality. These early
developments revolutionized the study of numerical methods for non-smooth
non-convex optimization problems. \\

In this work, we continue the abstract unification of the convergence analysis
of algorithms for non-smooth non-convex optimization \cite{ABS13,BST14}. Their
convergence analysis is driven by two central assumptions: a \emph{sufficient
decrease condition} and a \emph{relative error condition}. While they use the
sufficient decrease condition on the objective function, \cite{OCBP14}
formulates conditions that apply to a global surrogate function of
Lyapunov-type, which allows the objective values also to increase locally. Note
that this idea is different from the majorization minimization principle
\cite{HL04}, where in each iteration a majorizer of the objective is
constructed and minimized, which usually leads to a descent of the actual
objective values. In the KL context, this algorithmic strategy was used in
\cite{BP16,ODBP15}, and led to another abstract convergence result in
\cite{BP16} alike \cite{ABS13}. The abstract conditions formulated in our paper
contains \cite{ABS13,BST14,BP16,OCBP14,ODBP15} as special instances.

The relative error condition is justified by the fact that most algorithms
require to solve subproblems for which possibly inexact approaches
are required. The condition reflects \emph{relative inexact optimality
conditions} \cite{ABS13}, and is related to \cite{IPS03,SS99,SS99b,SS01}.
In \cite{BP16} the relative error condition is of explicit nature (see 
also \cite{AMA05,Noll13}), whereas in \cite{ABS13,BST14,OCBP14} it is implicit.
The abstract convergence theorem in our paper comprises the explicit and the 
implicit formulation.

The sufficient decrease condition and the relative error condition depend
rather on the structure of the algorithm than on fine properties of the
objective function. Therefore, the parameters appearing in these conditions are
tightly linked to properties of the algorithm such as the step size. While the
abstract convergence conditions discussed so far rely on a constant choice of
these parameters, Frankel et al. \cite{FGP14} introduced a significantly more
flexible parameter setting into these conditions. As a result, an alternating
version of the variable metric forward--backward splitting algorithm is
formulated and its convergence is proved, which opens the door for  non-smooth
and non-convex version of the Levenberg--Marquardt algorithm. The conditions 
in our paper are formulated such that \cite{FGP14} appears as a special case.

Beyond the flexibility introduced in \cite{FGP14}, in this paper, \ii1 we allow 
for a parametric function for which the sufficient decrease condition is 
required. This allows the objective or any surrogate relative to which decrease
is measured can change along the iterations. We believe that this additional 
flexibility has significant potential, which in this paper is only rudimentary
explored in the context of an inertial variable metric method. \ii2 The relative 
error condition can be formulated with respect to a linear combination of 
finitely many distance terms, which seems to be essential for multi-step 
methods \cite{OCBP14,Ochs15,BCL15,LFP16}. Finally, \ii3 all distances and the 
decrease in \ii1 are formulated using abstract distances. Of course, unless 
there is a closer relation between the abstract distance measure and the 
Euclidean metric, we have to content ourselves with a weaker convergence
result. Nevertheless, we consider this as an essential step to generalize the
convergence results further; possibly to algorithms that use Bregman distances
\cite{Bregman67} without smoothness or strong convexity assumption. In the
present paper, we use the abstract distance measures to restrict the Euclidean
distance to blocks of coordinates, which leads (almost for free) to a block
coordinate version of the inertial variable metric method iPiano. Without the 
variable metric aspect, the block coordinate inertial method was already proposed in
\cite{PS16}, though as a result of a more explicit analysis. \\

So far, we focused on abstract convergence results for non-smooth non-convex
optimization problems. As mentioned above, there are many concrete algorithms
that are proved to converge in such a general setting using the abstract 
conditions or an explicit verification of the convergence following the lines
of the abstract convergence proof. 

Convergence of the \emph{gradient method} is proved in \cite{AMA05,ABS13}, and
has been extended to \emph{proximal gradient descent} (resp. forward--backward
splitting method) \cite{ABS13}, which applies to a class of problems that is
given as the sum of a (possibly non-smooth and non-convex) function and a
smooth (possibly non-convex) function. Accelerations by means of a
\emph{variable metric} are considered in \cite{CPR13,FGP14}, and in combination
with a line-search procedure in \cite{BLPPR16}. The convergence of
\emph{proximal methods} is inspected in \cite{AB09,ABS13,BDLM10,MP10}, and an
\emph{alternating proximal method} is considered in \cite{ABRS10}. Extensions
to \emph{block coordinate} methods are given, e.g. in \cite{ABS13} under the
name regularized Gauss--Seidel method, which is actually a variable metric
version of the block coordinate methods in \cite{ABRS10,Auslender92,GS99}. The
combination of the ideas of alternating proximal minimization and
forward--backward splitting can be found in \cite{BST14}, where the algorithm
is called \emph{proximal alternating linearized minimization (PALM)}. For an
extension that allows the metric to change in each iteration with a flexible
order of the block iterations we refer to \cite{CPR16}.
Convergence of a non-smooth subgradient method is studied in \cite{Noll13,Hosseini15}.

Another possibility to accelerate descent methods (instead of using a variable
metric) are so-called \emph{inertial methods}. In convex optimization, some
inertial or overrelaxation methods are known to be optimal \cite{Nest04}.
Although it is hard to obtain sharp lower complexity bounds in the non-convex
setting, hence to argue about optimal methods, experiments show a favorable
performance of inertial algorithms. In \cite{OCBP14} an extension of inertial
gradient descent (also known as \emph{Heavy-ball method} or gradient descent
with momentum), which includes an additional non-smooth term in the objective
function alike forward--backward splitting, is analyzed in the KL framework.
The proposed algorithm is called \emph{iPiano} and shows good performance in
applications. An earlier subsequential convergence proof of Polyak's Heavy-ball
method \cite{Polyak64} without the KL inequality for smooth non-convex
functions is proposed in \cite{ZK93}.  In \cite{Ochs15,BCL15} the original
problem class \enquote{non-smooth convex plus smooth non-convex} in
\cite{OCBP14} was extended to \enquote{non-smooth non-convex plus smooth
non-convex}. In \cite{BCL15} also (smooth and strongly convex) Bregman
proximity functions are used in the update step. See \cite{BC15} for a variant
of this algorithm. A block coordinate version of iPiano or an \emph{inertial 
variant of the proximal alternating linearized minimization} method was recently
proposed as iPALM in \cite{PS16}. A variable metric version of iPiano and iPALM
---\emph{block coordinate variable metric iPiano}---is proposed in this paper.  The
accelerated method in \cite{LL15} is based on an \emph{extrapolation} of the gradient
alike Nesterov's proximal gradient method instead of an inertial term.  Liang
et al. \cite{LFP16} pursue a unifying approach of the preceding methods by a
\emph{generic multi-step method}.  All of these inertial methods share the property
that the sufficient decrease condition holds for a Lyapunov function instead of
the actual objective function.

This concept is important beyond inertial methods. It is used to prove 
convergence of splitting methods for \emph{composite problems} \cite{LP15b}, 
\emph{Douglas--Rachford splitting} \cite{LP16} and \emph{Peaceman--Rachford 
splitting} \cite{LP15a} for non-convex optimization problems.\\

Section~\ref{sec:prelim} introduces the basic notation and results from
(non-smooth) variational analysis \cite{Rock98} and the \KL inequality.
Section~\ref{sec:abstr-conv} formulates the basic conditions for the abstract
convergence theorem, which is motivated by the results in
\cite{ABS13,FGP14,OCBP14,BST14}. The gained flexibility of the conditions is
compared to related work in Section~\ref{sec:rel-abstr-conv}, and further
discussed in Section~\ref{subsec:discuss-cond-perspectives} where also some
future perspectives are provided. Examples for the necessity of the
generalizations are given in Appendix~\ref{sec:appdx}. The convergence
under the abstract conditions is proved in Section~\ref{subsec:conv-ana-abstr}.
The flexibility that is gained is used in Section~\ref{sec:vm-iPiano} to prove
convergence of a variable metric version of iPiano \cite{OCBP14,Ochs15} and in
Section~\ref{sec:bcvm-iPiano} of a block coordinate variable metric version of
iPiano. Several block coordinate, variable metric, and inertial versions of
forward--backward splitting/iPiano are applied to an image inpainting problem
in Section~\ref{sec:num}, which emphasizes the importance of a variable metric
and block coordinate methods.

%%%%%%%%%%%%%
%% SECTION %%
%%%%%%%%%%%%%
\section{Preliminaries} \label{sec:prelim}

\subsection{Notation and definitions}
Throughout this paper, we will always work in a finite dimensional Euclidean vector space $\R^\dimN$ of dimension $\dimN\in\N$, where $\N:=\{1,2,\ldots\}$. Define $\Z:=\set{\ldots,-1,0,1,\ldots}$. The vector space is equipped with the standard Euclidean norm $\norm\cdot := \norm[2]{\cdot}$ that is induced by the standard Euclidean inner product $\norm\cdot = \sqrt{\scal\cdot\cdot}$. If specified explicitly, we work in a metric induced by a symmetric positive definite matrix $A\in\spd(\dimN)\subset\R^{\dimN\times\dimN}$, represented by the inner product $\scal xy_A:= \scal{Ax}y$ and the norm $\norm[A]{x}:= \sqrt{\scal xx_A}$. For $A\in\spd(\dimN)$ we define $\varsigma(A)\in\R$ as the largest value that satisfies $\norm[A]{x}^2 \geq \varsigma (A) \norm[2]{x}^2$ for all $x\in\R^\dimN$.

As usual, we consider extended read-valued functions $\map f{\R^\dimN}{\eR}$,
$\eR:=\R\cup\set{+ \infty}$, that are defined on the whole space with
\emph{domain} given by $\dom f := \set{x\in\R^\dimN\vert\, f(x)< +\infty}$. A
function is called \emph{proper} if $\dom f\neq\emptyset$. We define the 
\emph{epigraph} of the function $f$ as $\epi f:=
\set{(x,\mu)\in\R^{\dimN+1}\vert\, \mu \geq f(x)}$. We will also need to
consider set-valued mappings $\smap{F}{\R^\dimN}{\R^\dimM}$ defined by the
\emph{graph} 
\[
  \Graph F := \set{(x,y) \in \R^\dimN\times\R^\dimM\vert\, y \in F(x)} \,,
\]
where the domain of a set-valued mapping is given by $\dom F:= \set{x \in \R^\dimN\vert\, F(x)\neq \emptyset}$. For a proper function $\map f{\R^\dimN}{\eR}$ we define the set of \emph{(global) minimizers} as 
\[
  \arg\min f := \arg\min_{x\in\R^\dimN} f := \set{x\in\R^\dimN\vert\, f(x) = \inf f}\,,\qquad \inf f := \inf_{x\in\R^\dimN} f(x) \,.
\]
The \emph{Fr\'echet subdifferential} of $f$ at $\bar x \in\dom f$ is the set $\rpartial f(\bar x)$ of those elements $v \in \R^\dimN$ such that
\[
  \liminf_{\substack{x\to \bar x\\ x\neq \bar x}} \frac{f(x) - f(\bar x) - \scal{v}{x-\bar x}}{\norm{x-\bar x}} \geq 0 \,.
\]
For $\bar x\not\in \dom f$, we set $\rpartial f(\bar x) = \emptyset$. 
For convenience, we introduce \emph{$f$-attentive convergence}: A sequence $\seq[n\in\N]{x^n}$ is said to \emph{$f$-converge} to $\bar x$ if 
\[
  x^n \to \bar x \quad\text{and}\quad f(x^n) \to f(\bar x) \quad\text{as } n\to\infty\,,
\]
and we write $x^n\fto \bar x$. The so-called \emph{(limiting) subdifferential} 
of $f$ at $\bar x\in\dom f$ is defined by
\[
  \partial f(\bar x) := \set{v\in \R^\dimN\vert\, \exists\, 
                          x^n \fto \bar x,\;v^n\in \rpartial f(x^n),\;v^n \to v} \,,
\]
and $\partial f(\bar x) = \emptyset$ for $\bar x \not\in \dom f$. A point 
$\bar x\in \dom f$ for which $0\in \partial f(\bar x)$ is a called a 
\emph{critical point} of \emph{stationary point}. As a direct consequence of
the definition of the limiting subdifferential, we have the following
closedness property: 
\[
  x^n \fto \bar x,\ v^n\to \bar v,\ \text{and for all } n\in\N\colon v^n \in \partial f(x^n)\quad \Longrightarrow\quad  \bar v\in \partial f(\bar x) \,.
\]
\cite[Ex. 8.8]{Rock98} shows that at a point $\bar x\in\R^\dimX$, for the sum
of an extended-valued function $g$ that is finite at $\bar x$ and a 
continuously differentiable (smooth) function $f$ around $\bar x$, it
holds that $\partial (g+f)(\bar x) = \partial g(\bar x) + \nabla f(\bar x)$.
Moreover for a function $\map f{\R^\dimN\times\R^\dimM}\eR$ with $f(x,y) =
f_1(x) + f_2(y)$ the subdifferential satisfies $\partial f(x,y) = \partial
f_1(x) \times \partial f_2(y)$ \cite[Prop. 10.5]{Rock98}.

Finally, the \emph{distance} of $\bar x\in\R^\dimN$ to a set $\omega\subset \R^\dimN$ as is given by $\dist(\bar x,\omega) := \inf_{x\in\omega}\, \norm{\bar x - x}$ and we introduce $\norm[-]{\partial f(\bar x)}:= \inf_{v\in\partial f(\bar x)} \norm{v} =\dist(0, \partial f(\bar x))$ what is known as the \emph{lazy slope} of $f$ at $\bar x$. Note that $\inf \emptyset := +\infty$ by definition. Furthermore, we have (see \cite{FGP14}):
\begin{LEM} \label{lem:lazy-slope-liminf}
If $x^\n \fto \bar x$  and $\liminf_{n\to\infty}\, \norm[-]{\partial f(x^\n)} = 0$, then $0\in \partial f(\bar x)$.
\end{LEM}
For a function $f$, we use the notation 
$[f<\mu]:= \set{x\in \R^\dimN\vert\, f(x)< \mu}$. Analogously, we use the same
notation for other conditions, for example, $[f\geq \mu]$, $[f=1]$, etc.

\subsection{The \KL property}

\begin{DEF}[\KL property / KL property]\label[DEF]{def:KL-property}
Let $\map f {\R^\dimN}{\eR}$ be an extended real valued function and let $\bar x\in\dom\partial f$. If there exists $\eta\in(0,\infty]$, a neighborhood $U$ of $\bar x$ and a continuous concave function $\map{\phi}{[0,\eta)}{\R_+}$ such that 
\[
  \phi(0)=0,\quad  \phi\in\cd 1( (0,\eta) ),\quad\text{and}\quad \phi^\prime(s)>0\text{ for all }s\in (0,\eta),
\]
and for all $x\in U\cap [f(\bar x) < f(x) < f(\bar x) + \eta]$ the \KL inequality
\begin{equation}\label{eq:KL-ineq}
  \phi^\prime(f(x)-f(\bar x)) \norm[-]{\partial f(x)} \geq 1
\end{equation}
holds, then the function has the \KL property at $\bar x$.

If, additionally, the function is lower semi-continuous and the property holds for each point in $\dom \partial f$, then $f$ is called a \KL function.
\end{DEF} 
Figure~\ref{fig:ex-KL-a}, which is taken from \cite{Ochs15}, shows the idea and the variables appearing in the definition of the KL property for a smooth function. For smooth functions (assume $f(\bar x)=0$), \eqref{eq:KL-ineq} reduces to $\norm{\nabla (\phi\circ f)} \geq 1$ around the point $\bar x$, which means that after reparametrization with a \emph{desingularization function} $\phi$ the function is sharp. \enquote{Since the function $\phi$ is used here to turn a singular region---a region in which the gradients are arbitrarily small---into a regular region, i.e. a place where the gradients are bounded away from zero, it is called a desingularization function for $f$.} \cite{ABS13}. It is easy to see that the KL property is satisfied for all non-stationary points \cite{ABRS10}.
\begin{figure}[t]
\begin{center}
  \newcommand{\thisTikzScaling}{1.0}
  %%%%%     \documentclass[11pt,a4paper,leqno]{report}
%%%%%     \usepackage[latin1]{inputenc}
%%%%%     \usepackage[ngerman]{babel}
%%%%%     \usepackage{amsthm}
%%%%%     \usepackage{amsmath}
%%%%%     \usepackage{amsfonts}
%%%%%     \usepackage{amssymb}
%%%%%     \usepackage{eurosym}
%%%%%     \usepackage{graphicx}
%%%%%     \usepackage{xcolor}
%%%%%     \usepackage{tikz}
%%%%%     \usetikzlibrary{patterns}
%%%%%     \usepackage{pgfplots}
%%%%%     \newcommand{\thisTikzScaling}{1.0}
%%%%%     \newcommand{\ncone}[1]{\ensuremath{N_{#1}}} % normal cone
%%%%%     \newcommand{\rncone}[1]{\ensuremath{\widehat N_{#1}}}
%%%%%     \newcommand{\epi}{\operatorname{epi}}
%%%%%     \newcommand{\hpartial}{\partial^\infty}
%%%%%     \newcommand{\rpartial}{\widehat\partial}
%%%%%     \usepackage{ifthen}
%%%%%     \newcommand{\set}[2][]{\ifthenelse{\equal{#1}{\empty}}{\lbrace #2 \rbrace}{\lbrace #2\vert\, #1\rbrace}}
%%%%%     
%%%%%     
%%%%%     
%%%%%     % shrink page margins % use larger page area
%%%%%     \usepackage[top=3cm, bottom=2cm, left=2cm, right=2cm]{geometry}
%%%%%     
%%%%%     
%%%%%     \begin{document}
%%%%%     
%%%%%     % remove page numbering on this page
%%%%%     \pagenumbering{gobble}
%%%%%     
\begin{tikzpicture}[scale=\thisTikzScaling]
\def\mycoordinatesystem{
  \begin{scope}[thick]
  \draw[-latex] (-0.5,0) -- (4.5,0);% node[below] {$x$};
  \draw[-latex] (0,-0.5) -- (0,3.5);% node[right] {$y$};
  \end{scope}
}
  \mycoordinatesystem

%%  \draw[help lines,step=0.2,draw=black!30] (-0.5,-1.5) grid (4.5,3.5);
  \def\tikzF#1{0.5*(#1)^2}

  \begin{scope}[shift={(2.5,1)}]
    \draw[domain=-2.2:-1.9,smooth,variable=\x,samples=200,thick] plot ({\x},{\tikzF{\x}});
    \draw[domain=-1.6:2.2,smooth,variable=\x,samples=200,thick] plot ({\x},{\tikzF{\x}});
    \node[right] at (2,{\tikzF{2}}) {$f$};
    \begin{scope}[thick, semitransparent]
      \draw[-latex] (-2,0) -- (2,0);% node[below] {$x$};
      \draw[-latex] (0,-0.5) -- (0,1.5);% node[right] {$y$};
      \node[left] at (0,1.5) {$f(x)-f(\bar x)$};
    \end{scope}

    \node[below left] at (-0.3,-0.3) {$(\bar x, f(\bar x))$};
    \draw[<-,>=stealth,shorten <=2pt] (0,0) to[bend left=15] (-0.45,-0.5);

    \draw[very thick,blue!50,dashed] (-1.7,0) -- (1.7,0);
    \node[left,blue!50] at (-2,-0.3) {$U$};
    \draw[->,>=stealth,blue!50,shorten >=1.5pt] (-2.1,-0.3) to[bend right=25](-1.5,0);
    \draw[very thick,blue!80!black] ({-sqrt(1.4)},0) -- ({sqrt(1.4)},0);
    \draw[thick,blue!80!black,fill=white] (0,0) circle [radius=1.5pt];
    \node[below right,blue!80!black] at (0,-1.3) {$U\cap [f(\bar x) < f(x) < f(\bar x) +\eta]$};
    \draw[->,>=stealth,blue!80!black,shorten >=1.5pt] (0.5,-1.3) to[bend left=15](0.5,0);

    \draw[magenta,dashed] (2,0.7) -- (-2,0.7) node[left] {$f(\bar x) + \eta\ $};

    \draw[domain=0:2,smooth,variable=\x,samples=200,thick,dashed,green!70!black] plot ({\x},{sqrt(\x)});
    \node[below right,green!70!black] at (2,{sqrt(2)}) {$\varphi$};
  \end{scope}
  
  \draw[->,>=stealth] (5,1.5) to[bend left=15] (7,1.5);

  \begin{scope}[shift={(9.5,1)}]
    \draw[domain=-1.6:2.0,smooth,variable=\x,samples=200,thick] plot ({\x},{1.2*sqrt(\tikzF{\x})});
    \node[right] at (2,{1.2*sqrt(\tikzF{2})}) {$\varphi\circ f$};
    \begin{scope}[thick, semitransparent]
      \draw[-latex] (-2,0) -- (2,0);% node[below] {$x$};
      \draw[-latex] (0,-0.5) -- (0,1.5);% node[right] {$y$};
      \node[below] at (2,0) {$x$};
    \end{scope}
    
    \draw[very thick,blue!80!black] ({-sqrt(1.4)},0) -- ({sqrt(1.4)},0);
    \draw[thick,blue!80!black,fill=white] (0,0) circle [radius=1.5pt];
    \draw[->,>=stealth,blue!80!black,shorten >=1.5pt] (-1.45,-1.5) to[bend right=15](0.5,0);

    \node[below left] at (-0.3,-0.3) {$(\bar x, f(\bar x))$};
    \draw[<-,>=stealth,shorten <=2pt] (0,0) to[bend left=15] (-0.45,-0.5);

  \end{scope}

\end{tikzpicture}
%%%%%     
%%%%%     \end{document}
%%%%%     
\end{center}
\caption{\label{fig:ex-KL-a}Example of the KL property for a smooth function. The composition $\phi\circ f$ has a slope of magnitude $1$ except at $\bar x$. }
\end{figure}
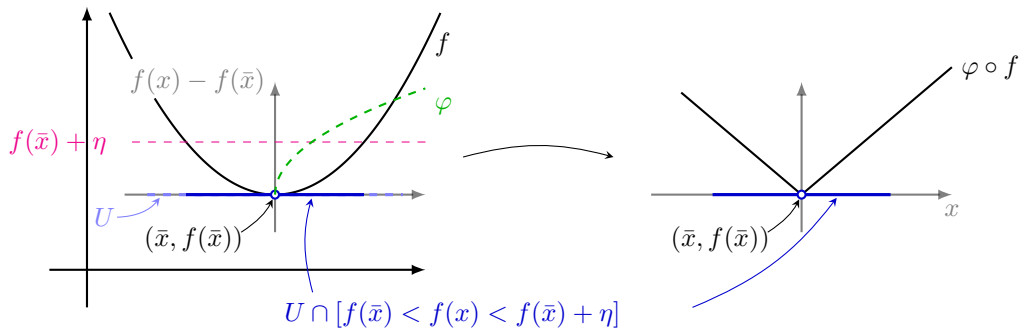

The KL property is satisfied by a large class of functions, namely functions that are definable in an $\o$-minimal structure (see \cite[Thm. 14]{ABRS10} and \cite[Thm. 14]{BDLS07}).
\begin{THM}[Nonsmooth \KL inequality for definable functions] \label[THM]{thm:nonsmooth-KL}
  Any proper lower semi-continuous function $\map{f}{\X}{\eR}$ which is definable in an $\o$-minimal structure $\structO$ has the \KL property at each point of $\dom \partial f$. Moreover the function $\phi$ in \cref{def:KL-property} is definable in $\structO$.
\end{THM}
In particular, semi-algebraic and globally subanalytic sets and functions are
definable in such a structure. There is even an $\o$-minimal structure that
extends the one of globally subanalytic functions with the exponential function
(thus also the logarithm is included) \cite{Wilkie96,Dries98}. In fact, 
$\o$-minimal structures can be seen as an axiomatization of
the nice properties of semi-algebraic functions, and are therefore designed
such that the structure is preserved under many operations, for example,
pointwise addition and multiplication, composition and inversion. A brief
summary of the concepts that are important for this paper can be found in
\cite{ABRS10}.

Before we introduce the general framework and the convergence analysis in the next sections, let us first consider a so-called \emph{uniformization results}, which was proved in \cite{AB09} for the \Loj property and adjusted in \cite{BST14} for the KL property. Its main implication for this paper---like in \cite{BST14}---is that it allows for a direct proof of the main convergence theorem without the need of an induction argument.
\begin{LEM}[Uniformization result \cite{BST14}] \label{lem:uniformization}
  Let $\omega$ be a compact set and let $\map{f}{\R^d}{\eR}$ be a proper and lower semi-continuous function. Assume that $f$ is constant on $\omega$ and satisfies the KL property at each point of $\omega$. Then, there exist $\eps>0$, $\eta>0$, and a continuous concave function $\map{\phi}{[0,\eta)}{\R_+}$ such that 
\[
  \phi(0)=0,\quad  \phi\in\cd 1( (0,\eta) ),\quad\text{and}\quad \phi^\prime(s)>0\text{ for all }s\in (0,\eta),
\]
  such that for all $\bar x\in\omega$ and all $x$ in the following intersection 
  \begin{equation}\label{eq:uniformized-KL-set}
    [\dist(x,\omega) < \eps] \cap [f(\bar x) < f(x) < f(\bar x) + \eta]
  \end{equation}
  one has,
  \[
    \phi^\prime(f(x) - f(\bar x)) \norm[-]{\partial f(x)} \geq 1\,.
  \]
\end{LEM}

%%%%%%%%%%%%%
%% SECTION %%
%%%%%%%%%%%%%
\section{An abstract inexact convergence theorem} \label{sec:abstr-conv}

In this section, let $\map{\F}{\R^\dimN\times \R^\dimP}{\eR}$ be a proper,
lower semi-continuous function that is bounded from below. We analyze
convergence of an abstract algorithm that generates a sequence
$\seq[n\in\N]{x^n}$ in $\R^\dimN$ under the following realistic assumptions.
Many algorithms, such as the gradient descent method, forward--backward
splitting, alternating projection, proximal minimization, Heavy-ball method,
iPiano, and many more methods satisfy these assumption. An application to
block coordinate and variable metric iPiano is presented in 
Sections~\ref{sec:vm-iPiano} and~\ref{sec:bcvm-iPiano}.
\setcounter{ASS}{7}
\begin{ASS}\label{ass:Hs}
  Let $\seq[\n\in\N]{u^\n}$ be a sequence of parameters in $\R^\dimP$, and let
  $\seq[n\in\N]{\eps_n}$ be an $\ell_1$-summable sequence of non-negative real
  numbers. Moreover, we assume there are sequences $\seq[n\in\N]{a_n}$,
  $\seq[n\in\N]{b_n}$, and $\seq[\n\in\N]{d_\n}$ of non-negative real numbers,
  a non-empty finite index set $I\subset\Z$ and $\theta_i\geq 0$, $i\in I$, with 
  $\sum_{i\in I}\theta_i = 1$ such that the following holds:
\begin{enumerate}[label=(H\arabic*),ref=(H\arabic*)]
\item\label{ass:Hs:descent} (Sufficient decrease condition) 
      For each $n\in \N$, it holds that
\[
     \F(x\iter\np,u\iter\np) + a\pit\n d\pit\n^2 \leq  \F(x\iter\n,u\iter\n)\,.
\]
\item\label{ass:Hs:error} (Relative error condition) 
      For each $n\in\N$, the following holds: (set $d\pit{j}=0$ for $j\leq0$)
\[
    b\pit{\np} \norm[-]{\partial \F(x\iter\np,u\iter\np) } 
    \leq b \sum_{i\in I} \theta\pit{i}d\pit{\np-i} + \eps\pit{n+1} \,.
\]
\item\label{ass:Hs:cont} (Continuity condition) 
      There exists a subsequence $\seq[j\in\N]{(x\iter{n_j},u\iter{n_j})}$ and 
      $(\tilde x,\tilde u)\in\R^\dimN\times\R^\dimP$ such that 
\[
  (x\iter{n_j},u\iter{n_j}) \Fto[\F] (\tilde x,\tilde u)\quad \text{as}
  \quad j\to\infty \,.
\]
\item\label{ass:Hs:distance} (Distance condition) 
      It holds that
\[
  d\pit\n \to 0 \Longrightarrow \norm[2]{x\iter\np-x\iter\n} \to 0 
  \qquad\text{and}\qquad
%  d\pit\n = 0 \Longrightarrow x\iter\np=x\iter\n \,,
  \exists \n^\prime\in\N\colon \forall \n\geq \n^\prime \colon d\pit\n = 0 
  \Longrightarrow 
  \exists \n^{\prime\prime}\in\N\colon \forall \n\geq \n^{\prime\prime} \colon 
      x\iter\np=x\iter\n 
\]
\item\label{ass:Hs:params} (Parameter condition) 
      It hold that
 \[
 \seq[\n\in\N]{b\pit\n}\not\in\ell_1\,, \quad  
 \sup_{n\in\N} \frac 1{b\pit\n a\pit\n} < \infty\,, \quad 
 \inf_\n a\pit\n =: \underline a > 0\,.
\]
\end{enumerate}
\end{ASS}
Let us first discuss how these assumptions generalize previous results and 
what are the perspectives of the newly gained flexibility. The convergence 
of the sequence $\seq[\n\in\N]{x\iter\n}$ is proved in 
Theorem~\ref{thm:KL-theorem-descent}. 

\subsection{Relation to other abstract convergence conditions} \label{sec:rel-abstr-conv}

The following works explicitly formulate abstract conditions that are used 
in specific algorithms. Examples of algorithms for which the generalizations are
necessary are provided in Appendix~\ref{sec:appdx}. 

\paragraph{Relation to \cite{ABS13}.} For a proper lower semi-continuous function $\map f{\R^N}{\eR}$
and a sequence $\seq[\n\in\N]{x^\n}$, the conditions in \cite{ABS13} are the 
following:
{\it
\begin{enumerate}[label=(ABS13-H\arabic*),ref=(ABS13-H\arabic*),%
                  leftmargin=*,align=left]
\item\label{ass:ABS13-Hs:descent} For each $\n\in \N$, 
  $f(x\iter\np) + a\norm[2]{x\iter\np - x\iter\n}^2 \leq f(x\iter\n)$\,.
\item\label{ass:ABS13-Hs:error} For each $\n\in\N$, the exists 
  $w\iter\np\in\partial f(x\iter\np)$ such that 
  $\norm{w\iter\np} \leq b\norm[2]{x\iter\np - x\iter\n}$.
\item\label{ass:ABS13-Hs:cont} There exists a subsequence 
  $\seq[j\in\N]{x\iter{\n_j}}$ and $\tilde x$ such that 
  $x\iter{\n_j} \to \tilde x$ and $f(x\iter{\n_j}) \to f(\tilde x)$ 
  as $j\to \infty$.
\end{enumerate}
}
If the conditions \ref{ass:ABS13-Hs:descent}--\ref{ass:ABS13-Hs:cont} hold, 
then also Assumption~\ref{ass:Hs} is satisfied, which shows that our result is 
more general. The relation is explicitly shown by setting
$\F(x\iter\n,u\iter\n) = f(x\iter\n)$, $u\iter\n=0$, $a\pit\n = a\in\R$, 
$b\pit\n=1$, $I=\set{1}$, $\theta_1=1$, $\eps\pit\n=0$ for all $\n\in\N$, and 
$d\pit\n = \norm[2]{x\iter\np-x\iter\n}$. 

\paragraph{Relation to \cite{FGP14}.} In \cite{FGP14}, the conditions in
\cite{ABS13} are generalized to a flexible parameter setting and 
Hilbert spaces. In $\R^N$, the conditions read as follows:
{\it
\begin{enumerate}[label=(FGP14-H\arabic*),ref=(FGP14-H\arabic*),%
                  leftmargin=*,align=left]
\item\label{ass:FGP14-Hs:descent} For each $\n\in \N$, for some $a\pit\n>0$,
  $f(x\iter\np) + a_n\norm[2]{x\iter\np - x\iter\n}^2 \leq f(x\iter\n)$\,.
\item\label{ass:FGP14-Hs:error} For each $\n\in\N$, for some $b\pit\np>0$ and 
  $\eps\pit\np\geq0$, 
  $b\pit\np\norm[-]{\partial f(x\iter\np)} 
    \leq \norm[2]{x\iter\np - x\iter\n} + \eps\pit\np$.
\item\label{ass:FGP14-Hs:cont} The sequences 
  $\seq[\n\in\N]{a\pit\n}$, $\seq[\n\in\N]{b\pit\n}$, $\seq[\n\in\N]{\eps\pit\n}$
  satisfy
  \[
    a\pit\n \geq \underline a > 0\ \;\text{for all}\;\  \n\in\N\,,\quad
    \seq[\n\in\N]{b\pit\n}\not\in\ell_1\,, \quad  
    \sup_{n\in\N} \frac 1{b\pit\n a\pit\n} < \infty\,, \quad \text{and}\quad
    \seq[\n\in\N]{\eps\pit\n}\in\ell_1\,.
  \]
\end{enumerate}
}
The continuity condition \ref{ass:ABS13-Hs:cont} is replaced by a 
$f$-precompactness assumption. The fact that Assumption~\ref{ass:Hs} is a 
generalization of these conditions follows immediately from
the relation to \cite{ABS13} and the design of our parameters 
$\seq[\n\in\N]{a\pit\n}$, $\seq[\n\in\N]{b\pit\n}$, $\seq[\n\in\N]{\eps\pit\n}$,
in analogy to those in \cite{FGP14}. Our relative error condition 
\ref{ass:Hs:error} and distance condition~\ref{ass:Hs:distance} are more general
and we allow for a second argument in the objective function $u^\n$ whose
convergence is not sought in the end, i.e., we allow for a controlled change of
the objective function along the iterations.

\paragraph{Relation to \cite{BP16}.} The abstract convergence statement
\cite[Proposition 4]{BP16}, poses conditions 
on a triplet of points $\set{x\iter\nm,x\iter\n,x\iter\np}$ and a function 
$\map{f}{\R^N}{\eR}$. The conditions are the following\footnote{We neglect the 
dependence of $b$ in \ref{ass:BP16-Hs:error} on the compact set that 
contains $x\iter\n$, as this set will be chosen to be the KL-neighborhood
of the set of limit points, which fixes the parameter for sufficiently large
$\n$.}:
{\it
\begin{enumerate}[label=(BP16-H\arabic*),ref=(BP16-H\arabic*),%
                  leftmargin=*,align=left]
\item\label{ass:BP16-Hs:descent} For each $\n\in \N$, 
  $f(x\iter\n) + a\norm[2]{x\iter\np - x\iter\n}^2 \leq f(x\iter\nm)$\,.
\item\label{ass:BP16-Hs:error} For each $\n\in\N$, 
  $\norm[-]{\partial f(x\iter\n)} \leq b \norm[2]{x\iter\np - x\iter\n}$.
\item\label{ass:BP16-Hs:cont} There exists a subsequence 
  $\seq[j\in\N]{x\iter{\n_j}}$ and $\tilde x$ such that 
  $x\iter{\n_j} \to \tilde x$ and $f(x\iter{\n_j}) \to f(\tilde x)$ 
  as $j\to \infty$.
\end{enumerate}
}
In contrast to \ref{ass:ABS13-Hs:error} and \ref{ass:FGP14-Hs:error}, the 
relative error condition \ref{ass:BP16-Hs:error} is explicit (like in 
\cite{AMA05} or more explicitly discussed in \cite[Section 2.4]{Noll13}), i.e., 
$x\iter\np$ does not appear inside the subdifferential estimate. Setting 
$d\pit\n = \norm[2]{x\iter{\n+2} - x\iter\np}$, $I=\set{1}$, 
$\theta_1=1$, $a\pit\n=a\in\R$, $b\pit\n=1$, $\eps\pit\n=0$, $u\iter\n=0$ and
$\F(x\iter\n,u\iter\n) = f(x\iter\n)$ in Assumption~\ref{ass:Hs}
recovers the conditions \ref{ass:BP16-Hs:descent}--\ref{ass:BP16-Hs:cont}.
Note that the definition of $d\pit\n$ does not conflict with \ref{ass:Hs:distance}.

\paragraph{Relation to \cite{OCBP14}.} The abstract convergence theorem of 
\cite{OCBP14} applies to a sequence $\seq[\n\in\N]{z\iter\n}$ given by 
$z\iter\n =(x\iter\n,x\iter\nm)$ with a sequence $\seq[\n\in\N]{x\iter\n}$ 
in $\R^\dimX$ for a function $\map f{\R^{2\dimX}}{\eR}$. The conditions
are the following:
{\it
\begin{enumerate}[label=(OCBP14-H\arabic*),ref=(OCBP14-H\arabic*),%
                  leftmargin=*,align=left]
\item\label{ass:OCBP14-Hs:descent} For each $\n\in \N$, 
  $f(z\iter\np) + a\norm[2]{x\iter\n - x\iter\nm}^2 \leq f(z\iter\n)$\,.
\item\label{ass:OCBP14-Hs:error} For each $\n\in\N$, the exists 
  $w\iter\np\in\partial f(z\iter\np)$\\ such that 
  $\norm{w\iter\np} \leq \frac b2(\norm[2]{x\iter\n - x\iter\nm} 
                            + \norm[2]{x\iter\np - x\iter\n})$.
\item\label{ass:OCBP14-Hs:cont} There exists a subsequence 
  $\seq[j\in\N]{z\iter{\n_j}}$ and $\tilde z$ such that 
  $z\iter{\n_j} \to \tilde z$ and $f(z\iter{\n_j}) \to f(\tilde z)$ 
  as $j\to \infty$.
\end{enumerate}
}
These conditions are recovered from our framework by setting 
$\F(z^n,u^\n) = f(z^\n)$, $d_n = \norm[2]{x^\n - x^\nm}$, $a_n = a\in\R$, $
b_n=1$, $I=\set{1,2}$, $\theta_{1} = \theta_2=\frac 12$  and $\eps_n=0$ for 
all $\n\in\N$. 
\begin{REM}
  Note that, using the equivalence between norms, the right hand side of the 
  inequality in \ref{ass:OCBP14-Hs:error} can be bounded from above:
  $\norm[2]{x\iter\n - x\iter\nm}+\norm[2]{x\iter\np - x\iter\n}
  \leq \sqrt{2} \norm[2]{z\iter\np-z\iter\n}$.
\end{REM}

\subsection{Discussion and perspectives} \label{subsec:discuss-cond-perspectives}

In Section~\ref{sec:rel-abstr-conv}, we have seen that the conditions in 
Assumption~\ref{ass:Hs} are more general than previous abstract convergence
results. In the following, we provide some discussion, intuition, and perspectives
of the conditions in Assumption~\ref{ass:Hs}.
\begin{itemize}
  \item Since $\F$ is bounded from below, \ref{ass:Hs:descent} requires that 
        $a\pit\n d\pit\n$ tends to $0$ as $\n\to\infty$. Moreover, as 
        $\inf_\n a\pit\n>0$, this implies that $d\pit\n\to 0$. 
  \item However, $\seq[\n\in\N]{a\pit\n}$ is not a priori assumed to be bounded. 
        The faster $a\pit\n$ tends to $\infty$, the faster the property 
        $a\pit\n d\pit\n\to0$ requires $d\pit\n$ to tend to $0$.
  \item If $d\pit\n\to 0$ and, assuming for a moment 
        that $\inf_\n b\pit\n>0$, \ref{ass:Hs:error} implies that 
        $\norm[-]{\partial \F(x\iter\n,u\iter\n)}\to 0$. However, 
        $\seq[\n\in\N]{b\pit\n}$ may tend to $0$, though not to fast because of
        \ref{ass:Hs:params}. The required slow behavior of $b\pit\n\to 0$
        will still allows us to conclude that 
        $\lim\inf_{\n\to\infty}\norm[-]{\partial \F(x\iter\n,u\iter\n)}=0$.
  \item The usage of the sequence $\seq{\eps\pit\n}$ accepts a larger relative
        error in \ref{ass:Hs:error} compared to \ref{ass:ABS13-Hs:error}.
  \item The sequence $\seq[\n\in\N]{d\pit\n}$ is introduced as a more general
        distance measure, which by \ref{ass:Hs:distance} is \enquote{consistent}
        with the Euclidean distance. The purpose of this generalization is to open
        the door for Bregman distances \cite{Bregman67} without the common 
        assumption of strong convexity or Lipschitz continuity of the gradient.
        Alternatively, the sequence $\seq[\n\in\N]{d\pit\n}$ can measure the
        distance between $\seq[\n\in\N]{x\iter\n}$ and a sequence of surrogate
        points, which only asymptotically, require
        $\norm[2]{x\iter\np-x\iter\n}\to 0$. Of course, when distances are only
        measured with such an abstract distance measure, convergence in the
        Euclidean sense cannot be expected without further assumptions. A third
        option, which we explore in this paper, is a sequence
        $\seq[\n\in\N]{d\pit\n}$ that measures the Euclidean distance only of a
        block of coordinates of $\seq[\n\in\N]{x\iter\n}$, which leads to block
        coordinate descent algorithms. A sufficient condition to achieve
        \ref{ass:Hs:distance} is to repeat each block after a finite
        number of steps (possibly unordered).
  \item The extension of \ref{ass:Hs:error} to the sum 
        $\sum_{i\in I} \theta_i d\pit{\n+1-i}$ seems to be important for 
        multi-step methods such as the Heavy-ball method \cite{Polyak64}, 
        iPiano \cite{OCBP14,Ochs15}, and other inertial forward--backward
        splitting methods \cite{BCL15,LFP16}. For the setting of \cite{LFP16},
        we provide some details in the appendix.
  \item The introduction of a sequence $\seq[\n\in\N]{u\iter\n}$ adds some 
        flexibility in the asymptotic behavior of the objective function. For 
        example, in \cite{OCBP14}, most of the analysis allows for step sizes
        and other parameters to change in each iteration. However, there is a 
        crucial parameter ($\delta$-parameter inside the Lyapunov function), 
        which is required to be constant for the convergence result. Using the
        gained flexibility from the sequence $\seq[\n\in\N]{u\iter\n}$, the 
        problem can be resolved. The variable metric iPiano considered in 
        Section~\ref{sec:vm-iPiano} requires a Lyapunov function that depends
        on a whole matrix, which thanks to the sequence $\seq[\n\in\N]{u\iter\n}$
        in Assumption~\ref{ass:Hs} can change in each iteration (see 
        \eqref{eq:hd-descent}). Note that this problem occurs due to the 
        definition of the Lyapunov function and does not appear, for example,
        in \cite{FGP14} where the variable metric is handled in a different 
        way.
\end{itemize}

\subsection{Convergence analysis} \label{subsec:conv-ana-abstr}

\subsubsection{Direct consequences of the descent property} 
Sufficient decrease \ref{ass:Hs:descent} of a certain quantity that can be 
related to the objective function value is key for the convergence analysis.
The following lemma lists a few simple but favorable properties for such
sequences.  
\begin{LEM} \label{lem:simple-conv}
  Let Assumption~\ref{ass:Hs} hold. 
  Then
  \begin{enumerate}
    \item\label{lem:F-monotone} $\seq[\n\in\N]{\F(x\iter\n,u\iter\n)}$ is 
          non-increasing,
    \item\label{lem:F-converges} $\seq[\n\in\N]{\F(x\iter\n,u\iter\n)}$ converges,
    \item\label{lem:simple-conv:iterates-converge} 
        $\sum_{\k=1}^\n d\pit\k^2 < +\infty$ and, therefore, $d\pit\n\to 0$ and 
        $\norm[2]{x\iter\np-x\iter\n} \to 0$, as $\n\to\infty$.
  \end{enumerate}
\end{LEM}
\begin{proof}
  \ref{lem:F-monotone} and \ref{lem:F-converges} follow from 
  \ref{ass:Hs:descent} and the boundedness from below of $\F$. 
  \ref{lem:simple-conv:iterates-converge} follows from summing 
  \ref{ass:Hs:descent} from $\k=1,\ldots,\n$ and \ref{ass:Hs:distance}, 
  \ref{ass:Hs:params}:
    \[
     \underline a \sum_{\k=1}^\n d\pit\k^2 
     \leq \sum_{\k=1}^\n a\pit\k d\pit\k^2 
     \leq \sum_{\k=1}^\n \F(x\iter\k,u\iter\k) - \F(x\iter\kp,u\iter\kp) 
     = ( \F(x\iter1,u\iter1) - \inf_{(x,u)\in\R^{\dimN}\times\R^\dimP}\F(x,u))  
     < +\infty \,. \qedhere
    \]
\end{proof}

\subsubsection{Direct consequences for the set of limit points}
Like in \cite{BST14}, we can verify some results about the set of limit points
(that depends on a certain initialization) of a bounded sequence
$\seq[\n\in\N]{(x\iter\n,u\iter\n)}$
\[
  \omega(x\iter0,u\iter0) := \limsup_{\n \to \infty}\, \set{(x\iter\n,u\iter\n)}  \,.
\]
This definition uses the outer set-limit of a sequence of singletons, which 
is the same as the set of cluster points in a different notation. Moreover, we 
denote by $\omega_{\F}(x\iter0,u\iter0)$ the subset of limit points that is
generated along $\F$-attentive subsequences, i.e.,
\[
  \omega_{\F}(x\iter0,u\iter0) := \set{(\bar x,\bar u)\in \omega(x\iter0,u\iter0)
  \setsep (x\iter{\n_j},u\iter{\n_j}) \Fto[\F] (\bar x,\bar u) \text{ for } j\to \infty} \,.
\]
We collect a few results that are of independent interest.
\begin{LEM} \label{lem:critical-point-set}
  Let Assumption~\ref{ass:Hs} hold and let $\seq[\n\in\N]{(x\iter\n,u\iter\n)}$ 
  be a bounded sequence.
  \begin{enumerate}
    \item\label{lem:critical-point-set:compcat}
      The set $\omega_{\F}(x\iter0,u\iter0)$ is non-empty and the set 
      $\omega(x^0,u^0)$ is non-empty and compact.
    \item\label{lem:critical-point-set:Fconstant}
      $\F$ is constant and finite on $\omega_{\F}(x\iter0,u\iter0)$.
  \end{enumerate}
\end{LEM}
\begin{proof}
\begin{enumerate}
  \item By \ref{ass:Hs:cont}, there exist a subsequence 
  $\seq[j\in\N]{(x\iter{\n_j},u\iter{\n_j})}$ of $\seq[\n\in\N]{(x\iter\n,u\iter\n)}$ 
  that converges to $(\tilde x,\tilde u)$, where at the same time the function 
  values along this subsequence converge to $\F(\tilde x,\tilde u)$, therefore 
  $\lim_{j\to\infty} (x\iter{\n_j},u\iter{\n_j}) \in \omega_{\F}(x\iter0,u\iter0)$ 
  and $\omega_{\F}(x\iter0,u\iter0)$ is non-empty. The non-emptiness of 
  $\omega(x\iter0,u\iter0)$ is clear and the compactness of $\omega(x\iter0,u\iter0)$ 
  is direct consequence of its definition as an outer set-limit and the 
  boundedness of $\seq[\n\in\N]{(x\iter\n,u\iter\n)}$.
  \item By Lemma~\ref{lem:simple-conv}\ref{lem:F-converges} 
  $\seq[\n\in\N]{\F(x\iter\n,u\iter\n)}$ converges to some $\tilde \F\in\R$. 
  For any $(\bar x,\bar u)\in \omega_{\F}(x\iter0,u\iter0)$ there exists a 
  subsequence $\seq[j\in\N]{(x\iter{\n_j},u\iter{\n_j})}$ that $\F$-converges to 
  $(\bar x,\bar u)$, therefore,
  \[
    \tilde \F = \lim_{j\to\infty} \F(x\iter{\n_j}, u\iter{\n_j}) 
              = \F(\bar x,\bar u)\,,
  \]
  which shows that $\F$ is constant on $\omega_{\F}(x\iter0,u\iter0)$.
\end{enumerate}
\end{proof}

\begin{LEM} \label{lem:critical-point-set-additional-unused}
  Let Assumption~\ref{ass:Hs} hold and $\seq[\n\in\N]{(x\iter\n,u\iter\n)}$ 
  be a bounded sequence. Denote by 
  $\Pi_x(\omega)=\set{x\in \R^\dimN\vert\, (x,u) \in \omega}$ 
  the projection of $\omega\in\R^{\dimN}\times\R^{\dimP}$ onto the first $\dimN$ 
  coordinates. Then, we have the following results:
  \begin{enumerate}
    \item\label{lem:critical-point-set:connected} 
          The set $\Pi_x\left(\omega(x\iter0,u\iter0)\right)$ is connected. 
    \item\label{lem:critical-point-set:connected-both} 
          If $\seq[\n\in\N]{u\iter\n}$ converges, then the set 
          $\omega(x\iter0,u\iter0)$ is connected. 
    \item\label{lem:critical-point-set:attraction} 
      It holds that 
      \[
        \lim_{\n\to\infty}\, \dist((x\iter\n,u\iter\n), \omega(x\iter0,u\iter0)) 
        = 0\,.
      \]
  \end{enumerate}
\end{LEM}
\begin{proof}
\ref{lem:critical-point-set:connected} is a simple application of the 
connectedness results \cite[Lemma 5]{BST14} and the fact that
$\norm[2]{x\iter\np-x\iter\n} \to 0$ for $\n\to\infty$ by 
Lemma~\ref{lem:simple-conv}\ref{lem:simple-conv:iterates-converge}.
\ref{lem:critical-point-set:connected-both} follows in almost the same manner,
as convergence of $u^\n$ implies $\norm[2]{u\iter\np - u\iter\n}\to 0$ as $\n
\to\infty$. \ref{lem:critical-point-set:attraction} is a direct consequence of
the definition of the set of limit points.
\end{proof}

\begin{LEM} \label{lem:critical-point-set-additional}
  Let Assumption~\ref{ass:Hs} hold, let $\seq[\n\in\N]{(x\iter\n,u\iter\n)}$ be 
  a bounded sequence and let $\sum_{\n=0}^\infty d\pit\n < \infty$. Then, the 
  set $\omega_{\F}(x\iter0,u\iter0)\subset \crit \F$.
\end{LEM}
\begin{proof}
  Let $(\bar x,\bar u) \in \omega(x\iter0,u\iter0)$. Then, since 
  $\seq[\n\in\N]{b\pit\n}\not\in\ell_1$ holds, from \ref{ass:Hs:error}, 
  $\seq[\n\in\N]{\eps\pit\n}\in\ell_1$ and
  \[
    \sum_{\n=0}^\infty b\pit\n \norm[-]{\partial \F(x\iter\n, u\iter\n)} 
    \leq b \sum_{\n=0}^\infty \sum_{i\in I}\theta\pit{i} d\pit{\n-i} 
         + \sum_{\n=0}^\infty \eps\pit\n 
    < \infty 
  \]
  follows $\liminf_{n\to\infty} \norm[-]{\partial \F(x\iter\n, u\iter\n)}=0$.
  For $(\bar x,\bar u) \in \omega_{\F}(x\iter0,u\iter0)$ the subsequence 
  $\seq[j\in\N]{(x\iter{\n_j},u\iter{\n_j})}$ $\F$-converges to $(\bar x,\bar u)$ 
  as $j\to\infty$ and Lemma~\ref{lem:lazy-slope-liminf} implies that 
  $0 \in \partial \F(\bar x, \bar u)$, which was to be proved.
\end{proof}
\begin{COR}
  Let Assumption~\ref{ass:Hs} hold and let $\seq[\n\in\N]{(x\iter\n,u\iter\n)}$ 
  be a bounded sequence. Suppose $\F$ is continuous on the set $W\cap \dom\F$ 
  with an open set $W\supset \omega(x\iter0,u\iter0)$ (e.g. $\F$ is continuous 
  on $\dom\F$), then
  \[
  \omega(x\iter0,u\iter0)=\omega_{\F}(x\iter0,u\iter0) \,.
  \]
\end{COR}
\begin{proof}
Let $(x\iter{\n_j},u\iter{\n_j}) \to (\bar x, \bar u) \in \omega(x\iter0,u\iter0)$ 
as $j\to \infty$. There is a neighborhood $V\subset W$ with 
$(\bar x, \bar u)\in V$ such that $(x\iter{\n_j},u\iter{\n_j})\in V \cap\dom \F$ 
for sufficiently large $j\in\N$ and continuity of $\F$ implies 
$(x\iter{\n_j},u\iter{\n_j}) \Fto[\F] (\bar x, \bar u)$, thus 
$\omega(x\iter0,u\iter0)\subset  \omega_{\F}(x\iter0,u\iter0)$. The converse 
inclusion holds by definition.
\end{proof}

\subsubsection{The convergence theorem}

\begin{THM}\label{thm:KL-theorem-descent}
  Suppose $\F$ is a proper lower semi-continuous \KL function that is bounded 
  from below. Let $\seq[\n\in\N]{x\iter\n}$ be a bounded sequence generated by 
  an abstract algorithm parametrized by a bounded sequence 
  $\seq[\n\in\N]{u\iter\n}$ that satisfies Assumption~\ref{ass:Hs}. Assume that
  $\F$-attentive convergence holds along converging subsequences of 
  $\seq[\n\in\N]{(x\iter\n,u\iter\n)}$, i.e.
  $\omega(x\iter0,u\iter0)=\omega_{\F}(x\iter0,u\iter0)$. 
  Then, the following holds:
  \begin{enumerate}
    \item\label{thm:KL-theorem-finite-length-abstr-dist} 
          The sequence $\seq[\n\in\N]{d\pit\n}$ satisfies 
          \begin{equation} \label{eq:KL-theorem-finite-length-abstr-dist}
            \sum_{\k=0}^{\infty} d\pit\k < +\infty\,,
          \end{equation}
          i.e., the trajectory of the sequence $\seq[\n\in\N]{x\iter\n}$ has 
          finite length with respect to the abstract distance measures 
          $\seq[\n\in\N]{d\pit\n}$.
    \item\label{thm:KL-theorem-finite-length-dist} 
          Suppose $d\pit\k$ satisfies 
          $\norm[2]{x\iter\kp-x\iter\k}\leq \bar c d\pit{\k+\k^\prime}$ for some 
          $\k^\prime\in\Z$ and $\bar c\in\R$, then
          \begin{equation} \label{eq:KL-theorem-finite-length-dist}
            \sum_{\k=0}^{\infty} \norm[2]{x\iter\kp-x\iter\k} < +\infty\,,
          \end{equation}
          and the trajectory of the sequence $\seq[\n\in\N]{x\iter\n}$ has 
          a finite Euclidean length, and thus $\seq[\n\in\N]{x\iter\n}$ 
          converges to $\tilde x$ from \ref{ass:Hs:cont}.
     \item\label{thm:KL-theorem-critical-point} 
     Moreover, if $\seq[\n\in\N]{u^\n}$ is a converging sequence, then 
     each limit point of $\seq[n\in\N]{(x^\n,u^\n)}$ is a critical point,
     which in the situation of \ref{thm:KL-theorem-finite-length-dist} is
     the unique point $(\tilde x, \tilde u)$ from \ref{ass:Hs:cont}.
  \end{enumerate}
\end{THM}
\begin{proof}
  By \ref{ass:Hs:cont} there exists a subsequence 
  $\seq[j\in\N]{(x\iter{\n_j},u\iter{\n_j})}$ such that 
  $(x\iter{\n_j},u\iter{\n_j})\Fto[\F] (\tilde x,\tilde u)$ as $j\to\infty$. 
  If there is $\n^\prime$ such that $\F(x\iter{\n^\prime},u\iter{\n^\prime}) 
  = \F(\tilde x,\tilde u)$, then \ref{ass:Hs:descent} implies that 
  $\F(x\iter\n,u\iter\n)=\F(\tilde x,\tilde u)$ for all $\n\geq \n^\prime$, 
  thus also $a\pit\n d\pit\n^2 = 0$ and by $\underline a>0$ (see 
  \ref{ass:Hs:distance}) $d\pit\n = 0$ for all $\n\geq \n^\prime$. 
  Therefore, \ref{ass:Hs:distance} shows that 
  $x\iter\np = x\iter\n$ for all $\n\geq \n^{\prime\prime}$ for some 
  $\n^{\prime\prime}\in\N$, and by induction 
  $\seq[\n\in\N]{x\iter\n}$ gets stationary (i.e. $x\iter\n=x\iter{\n^{\prime\prime}}$ 
  for all $\n\geq \n^{\prime\prime}$) and the statement is obvious.
  
  Now, we can assume that $\F(x\iter\n,u\iter\n) > \F(\tilde x,\tilde u)$ for 
  all $\n\in \N$. Moreover, non-increasingness of 
  $\seq[\n\in\N]{\F(x\iter\n,u\iter\n)}$ by \ref{ass:Hs:descent} implies that 
  for all $\eta>0$ there exists $\n_1\in\N$ such that 
  $\F(\tilde x,\tilde u) < \F(x\iter\n,u\iter\n) < \F(\tilde x,\tilde n) + \eta$ 
  for all $\n\geq \n_1$. By definition there is also a region of attraction for 
  the sequence $\seq[\n\in\N]{x\iter\n,u\iter\n}$, i.e., for all $\eps > 0$ 
  there exists $\n_2\in\N$ such that 
  $\dist((x\iter\n,u\iter\n), \omega(x\iter0,u\iter0)) < \eps$ 
  holds for all $\n\geq \n_2$. In total, we know that for all 
  $\n\geq \n_0 := \max\set{\n_1,\n_2}$ the sequence 
  $\seq[\n\in\N]{(x\iter\n,u\iter\n)}$ lies in the set
  \[
    [\F(\tilde x,\tilde u) < \F(x,u) < \F(\tilde x,\tilde u) + \eta] 
    \cap [\dist((x,u), \omega(x\iter0,u\iter0))< \eps] \,.
  \]
  
  Combining the facts that $\omega(x\iter0,u\iter0)=\omega_{\F}(x\iter0,u\iter0)$ 
  is nonempty and compact from 
  Lemma~\ref{lem:critical-point-set}\ref{lem:critical-point-set:compcat} with 
  $\F$ being finite and constant on $\omega(x\iter0,u\iter0)$ from 
  Lemma~\ref{lem:critical-point-set}\ref{lem:critical-point-set:Fconstant}, 
  allows us to apply Lemma~\ref{lem:uniformization} with 
  $\omega=\omega(x\iter0,u\iter0)$. Therefore, there are $\phi$, $\eta$, $\eps$ 
  as in Lemma~\ref{lem:uniformization} such that for $\n>\n_0$ 
  \begin{equation}\label{eq:uniform-KL-for-proof}
    \phi^\prime(\F(x\iter\n,u\iter\n)-\F(\tilde x,\tilde u))
        \norm[-]{\partial \F(x\iter\n,u\iter\n)} \geq 1
  \end{equation}
  holds on $\omega$. Plugging \ref{ass:Hs:error} into 
  \eqref{eq:uniform-KL-for-proof} yields
  \begin{equation}\label{eq:lower-bound-error}
    \phi^\prime(\F(x\iter\n,u\iter\n)-\F(\tilde x,\tilde u)) 
      \geq  b\pit\n \left( b\sum_{i\in I} \theta_id\pit{\n-i} + \eps\pit{n} \right)^{-1}\,.
  \end{equation}
  By concavity of $\phi$: (let $m>\n$)
  \[
     D_{\n,m}^\phi := \phi(\F(x\iter\n,u\iter\n)-\F(\tilde x,\tilde u)) 
                   - \phi(\F(x\iter{m},u\iter{m})-\F(\tilde x,\tilde u))  
     \geq \phi^\prime(\F(x\iter\n,u\iter\n)-\F(\tilde x,\tilde u)) 
              (\F(x\iter\n,u\iter\n) - \F(x\iter{m},u\iter{m}))\,,
  \]
  using \eqref{eq:lower-bound-error} and \ref{ass:Hs:descent}, we infer
  \[
    D_{n,n+1}^\phi \geq \frac{b\pit\n a\pit\n d\pit\n^2}
                             {b\sum_{i\in I}\theta_id\pit{\n-i} + \eps_{n}}  
    \quad \Leftrightarrow \quad 
    d\pit\n^2 
    \leq \left( \sum_{i\in I}\theta_i^\prime d\pit{\n-i} 
      + \eps\pit\n^\prime \right) 
          \left(\frac{b^\prime}{a\pit\n b\pit\n}D_{\n,\np}^\phi\right)
  \]
  where we use the substitutions $\overline{\theta}:= \sum_{j\in I}\theta_j$, 
  $b^\prime:=b\overline{\theta}$, 
  $\theta_i^\prime :=\theta_i / \overline{\theta}$, and
  $\eps\pit\n^\prime:= \eps\pit\n/b^\prime$.
  Applying $2 \sqrt{\alpha\beta} \leq \alpha + \beta$ for all 
  $\alpha, \beta \geq 0$, we obtain 
  (set $c:=\sup_\n \frac {b^\prime}{a\pit\n b\pit\n}<\infty$ 
  (by \ref{ass:Hs:distance}))
  \[
    2 d\pit\n 
    \leq \frac{b^\prime}{a\pit\n b\pit\n} D_{\n,\np}^\phi 
       + \sum_{i\in I}\theta_i^\prime d\pit{\n-i} + \eps\pit\n^\prime  
    \leq c D_{\n,\np}^\phi 
       + \sum_{i\in I}\theta_i^\prime d\pit{\n-i} + \eps\pit\n^\prime  \,.
  \]
  Now summing this inequality from $\k=\n_0,\ldots, \n$ yields:
  \begin{equation} \label{eq:thm:KL-theorem-descent-A}
      2 \sum_{k=n_0}^{n} d\pit\k 
      \leq 
      \sum_{\k=\n_0}^{\n} \sum_{i\in I}\theta_i^\prime d\pit{\k-i} 
          + c \sum_{\k=\n_0}^{\n} D_{\k,\kp}^\phi 
          + \sum_{\k=\n_0}^{\n} \eps\pit\k^\prime \,.
  \end{equation}
  The first sum on the right hand side can be rewritten as follows\footnote{%
    We use the convention that the summation is zero when the start index 
    is larger than the termination index.}: (use the substitution $j=\k-i$) 
  \[
    \sum_{\k=\n_0}^{\n} \sum_{i\in I}\theta_i^\prime d\pit{\k-i}
    = \sum_{i\in I} \sum_{j=\n_0-i}^{\n-i} \theta_i^\prime d\pit{j} 
    = \Big(\sum_{i\in I} \theta_i^\prime\Big)\sum_{j=n_0}^{\n} d\pit{j}
    + \sum_{i\in I} \sum_{j=\n_0-i}^{\n_0-1} \theta_i^\prime d\pit{j}
    + \sum_{i\in I} \sum_{j=\n+1}^{\n-i} \theta_i^\prime d\pit{j}\,.
  \]
  Using $\sum_{i\in I}\theta_i^\prime=1$ and rearranging terms in 
  \eqref{eq:thm:KL-theorem-descent-A} yields
  \[
    \sum_{k=n_0}^{n} d\pit\k 
    \leq
    \sum_{i\in I} \sum_{j=\n_0-i}^{\n_0-1} \theta_i^\prime d\pit{j}
      + \sum_{i\in I} \sum_{j=\n+1}^{\n-i} \theta_i^\prime d\pit{j}
      + c \sum_{\k=\n_0}^{\n} D_{\k,\kp}^\phi
      + \sum_{\k=\n_0}^{\n} \eps\pit\k^\prime \,.
  \]
  From this inequality, we conclude that 
  $\lim_{\n\to\infty} \sum_{\k=0}^\n d\pit\k<+\infty$. The first and second 
  term of the right hand side are finite summations and $d\pit\n\to 0$ as 
  $\n\to\infty$. The third term equals $cD_{n_0,n+1}^\phi$, which is bounded 
  from above by $\phi(\F^{n_0}(x^{n_0})-\F(\tilde x))<+\infty$. The last term
  is finite by assumption $\seq[n\in\N]{\eps_n}\in \ell_1$, which, in total,
  verifies \ref{thm:KL-theorem-finite-length-abstr-dist}.
  
  \ref{thm:KL-theorem-finite-length-dist} is a consequence of 
  \ref{thm:KL-theorem-finite-length-abstr-dist} and the fact that
  for arbitrary $m>n>0$
  \[
    \norm[2]{x\iter{m}-x\iter{\n}} 
    \leq \sum_{\k=\n}^{m-1} \norm[2]{x\iter\kp-x\iter\k}
    \leq c \sum_{\k=\n}^{m-1} d\pit{\k+\k^\prime} < + \infty 
  \]
  holds, which shows that $\seq[\n\in\N]{x\iter\n}$ is a Cauchy sequence 
  (The right hand side vanishes for $\n,m\to\infty$). Therefore, 
  $x\iter\n\to\tilde x$ as $\n\to\infty$, which verifies 
  \ref{thm:KL-theorem-finite-length-dist}.  
  Using \ref{thm:KL-theorem-finite-length-abstr-dist} and 
  \ref{thm:KL-theorem-finite-length-dist}, 
  \ref{thm:KL-theorem-critical-point} is a direct consequence of 
  Lemma~\ref{lem:critical-point-set-additional}.
\end{proof}

%%%%%%%%%%%%%%%%%
%% New Section %%
%%%%%%%%%%%%%%%%%
\section{Variable metric iPiano} \label{sec:vm-iPiano}

We consider a structured non-smooth, non-convex optimization problem with a 
proper lower semi-continuous extended valued function $\map{h}{\R^N}{\eR}$, 
$N\geq 1$, that is bounded from below by some value $\underline h > -\infty$:
\begin{equation}\label{eq:problem-class}
\min_{x \in \R^N}\, h(x)\,, \quad h(x) = f(x) + g(x)\,.
\end{equation}
The function $\map f{\R^N}{\R}$ is assumed to be $\cd 1$-smooth (possibly
non-convex) with $L$-Lipschitz continuous gradient on $\dom g$, $L>0$. Further,
let the function $\map g {\R^N}{\eR}$ be simple (possibly non-smooth and
non-convex) and prox-bounded, i.e., there exists $\lambda>0$ such that 
\[
  e_\lambda g(x) := \inf_{y\in\R^N} g(y) + \frac{1}{2\lambda}\norm{y-x}^2>-\infty
\]
for some $x\in\R^N$. Saying \enquote{$g$ is simple} refers to the fact that 
the associated proximal map can be solved efficiently for the global optimum. \\

We propose Algorithm~\ref{alg:ipiano-vm-gen} to find a critical point 
$x^*\in \dom h$ of $h$, which in this case is characterized by
\[
-\nabla f(x^*) \in \partial g(x^*)\,,
\]
where $\partial g$ denotes the limiting subdifferential. The parameter 
restrictions are discussed in Lemma~\ref{lem:vmipiano:step-size-necessary} and 
Remark~\ref{rem::ipiano-vm-param-discuss}.\\

Depending on the properties of $g$, the step size parameter $\alpha_\n$ and 
the inertial parameter $\beta_\n$ must satisfy different conditions. We analyse
the properties when $g$ is convex, semi-convex, or non-convex in a concise 
manner. If $g$ is semi-convex with respect to the metric induced by
$A\in\spd(\dimN)$, let $m$ be the semi-convexity parameter, i.e., $m\in \R$ is
the largest value such that $g(x) - \frac m2 \norm[A]{x}^2$ is convex. For
convex functions $m=0$ and for strongly convex functions $m>0$. Instead of 
considering the situation where $g$ is non-convex as a semi-convex function 
with \enquote{$m=-\infty$}, we introduce a \enquote{flag variable} 
$\sigma\in\set{0,1}$, which is $1$ if $g$ is semi-convex and $0$ if $g$ is 
non-convex. Note that if $\sigma =1$ the property of semi-convexity is satisfied 
for any $A\in\spd(\dimN)$, but with possibly changing modulus. Therefore, 
sometimes the metric is not explicitly specified. 

\begin{figure}[h!]
\centering
\setlength{\fboxrule}{1pt}
\fbox{
\begin{minipage}{0.95\textwidth}
\begin{ALG}\label{alg:ipiano-vm-gen}
\
\textbf{V}ariable \textbf{m}etric \textbf{i}nertial \textbf{p}rox\textbf{i}mal \textbf{a}lgorithm
for \textbf{n}onconvex \textbf{o}ptimization (vmiPiano)
\begin{itemize}
\item \textbf{Parameter}: Let
\begin{itemize}
  \item $\seq[n\in\N]{\alpha_n}$ be a sequence of positive step size parameters,
  \item $\seq[n\in\N]{\beta_n}$ be a sequence of non-negative parameters, and
  \item $\seq[\n\in\N]{A_\n}$ be a sequence of matrices $A_\n\in\spd(\dimN)$ 
        such that $A_n \preceq \opid$ and $\inf_n \varsigma (A_n) > 0$.
  \item Let $\sigma = 1$ if $g$ is semi-convex and $\sigma= 0$ otherwise.
\end{itemize}
\item \textbf{Initialization}: Choose a starting point $x^0 \in \dom h$ and 
                               set $x^{-1}= x^0$.
\item \textbf{Iterations $(n\ge 0)$}:  Update:
  \begin{equation}\label{eq:ipiano-vm-gen-up}
    \begin{split}
    y^{\n} =&\ x^{\n} + \beta_{\n}(x^{\n} - x^{\nm})  \\
    x^{\np} \in&\ \arg\min_{x\in\R^N}\ Q^n (x;x^n)\,, 
    \quad 
    Q^n (x;x^n) := g(x) + \scal{\nabla f(x^\n)}{x-x^\n} 
                 + \frac{1}{2\alpha_n} \norm[A_n]{x - y^\n}^2  \,,
    \end{split}
  \end{equation}
  where $L_n>\sigma m_n$ is determined such that 
  \begin{equation} \label{eq:ipiano:quad-upper-bnd}
    f(x^{n+1}) \leq f(x^n) + \scal{\nabla f(x^n)}{x^\np - x^{n}} + 
    \frac {L_n}{2} \norm[A_n]{x^\np-x^n}^2
  \end{equation}
  holds and $\alpha_n,\beta_n$ with $\inf_n \alpha_n>0$ are chosen such that 
  (see e.g. Lemma~\ref{lem:vmipiano:step-size-necessary})
    \begin{equation} \label{eq:ipiano-delta-gamma}
      \delta_n^\sigma:= \frac 12\left(\frac{1+\sigma - \beta_n}{\alpha_n} 
                      - (L_n-\sigma m_\n) \right)\quad\text{and}\quad
      \gamma_n := \delta_n^\sigma - \frac{\beta_n}{2\alpha_n} 
    \end{equation}
  satisfy 
  \begin{equation} \label{eq:ipiano-delta-gamma-cond}
    \inf_\n \gamma_\n>0\quad\text{and}\quad 
    \delta_{\np}^\sigma \norm[A_\np]{x^\np - x^n}^2 
        \leq \delta_n^\sigma \norm[A_n]{x^\np - x^n}^2  \,,
  \end{equation}
  where $m_n\in\R$ denotes the semi-convexity modulus of $g$ w.r.t. 
  $A_\n\in\spd(\dimN)$ (if $\sigma=1$).
  \end{itemize}
\end{ALG}
\end{minipage}
}
\end{figure}
\begin{LEM} \label{lem:vmipiano:step-size-necessary}
  A necessary condition for the sequences $\seq[\n\in\N]{\alpha_\n}$ and 
  $\seq[\n\in\N]{\beta_\n}$ to satisfy $\gamma_n \geq c >0$ for all $n\in\N$ is
  \[
      \alpha_n \leq \frac{1+\sigma - 2\beta_n}{L_n - \sigma m_n + 2c}
      \quad\text{and}\quad 
      \beta_n \leq \frac {1+\sigma}{2} \,.
  \]
\end{LEM}
\begin{proof}
  The bounds directly follow from $\inf_n \gamma_\n > 0$.
\end{proof}

\begin{REM}
The minimization problem in \eqref{eq:ipiano-vm-gen-up} is equivalent to 
(constant terms are dropped)
\begin{equation} \label{eq:ipiano-gen-up-variant-a}
  \arg\min_{x\in\R^N}\ g(x) 
    + \scal{\nabla f(x^\n)}{x-x^\n}
    -\frac{\beta_n}{\alpha_n}\scal{x^\n-x^\nm}{x - x^\n}_{A_n} 
    + \frac{1}{2\alpha_n} \norm[A_n]{x - x^\n}^2  \,.
\end{equation}
\end{REM}

The optimality condition of the minimization problem in 
\eqref{eq:ipiano-vm-gen-up} yields
\[
  0 \in \partial Q^n(x;x^n) = \partial g(x) + \nabla f(x^\n) + \frac {1}{\alpha_n} A_n(x-y^\n)
\]
and using the expression for $y^\n$ and a simple rearrangement, we obtain the
necessary condition for $x^\np$:
\begin{equation} \label{eq:ipiano-vm-gen-up-prox}
  x\in (\opid + \alpha_n A_n^{-1} \partial g)^{-1}
    \left( x^\n - \alpha_n A_n^{-1} \nabla f(x^\n) + \beta_n(x^\n - x^\nm)\right) \,.
\end{equation}
For a convex function $g$, inverting the expression 
$\opid + \alpha_n A_n^{-1} \partial g$ 
yields a unique solution and the inclusion can be replaced by an equality. 
Here, the operator is set-valued. 
\begin{REM} \label{rem::ipiano-vm-param-discuss}
\begin{itemize}
  \item The assumption in \eqref{eq:ipiano:quad-upper-bnd} is satisfied for 
  example, if $f$ has an $L$-Lipschitz continuous gradient with $A_n = \opid$, 
  or when a local estimate of the Lipschitz constant $L_n$ is known (also 
  $A_n = \opid$).
  \item Since $\nabla f$ is assumed to be Lipschitz continuous, given 
  $A\in\spd(\dimN)$, we can always find $L$ such that $A_n$ can be 
  \enquote{normalized} to $0\preceq A \preceq \opid$. In practice the algorithm 
  can be extended by a backtracking procedure for estimating $L_n$.
  \item The additional hyperparameters $\delta^\sigma_\n$ and $\gamma_n$ can be 
  seen as an disadvantage, however, actually, they allow for a constructive 
  selection of the step size parameters (cf. \cite{OCBP14}). For example in 
  \cite{BCL15}, such hyperparameters do not appear and only existence of 
  parameters that satisfy certain conditions can be guaranteed.
  \item The first condition in \eqref{eq:ipiano-delta-gamma-cond} is satisfied by 
  the parameter choice suggested in Lemma\ref{lem:vmipiano:step-size-necessary}. 
  The second condition can be achieved by specifying a monotonically 
  non-increasing sequence $\seq[\n\in\N]{\delta^\sigma_\n}$; then the condition
  on the descent w.r.t. the metric is slightly more restrictive than the
  standard assumption in this context \cite{CV13,CV14}, but could potentially
  be included into the backtracking procedure for
  \eqref{eq:ipiano:quad-upper-bnd}.
  \item Unlike in \cite{OCBP14,Ochs15}, where the sequence $\delta_\n$ is
  assumed to be stationary after a finite number of iterations to obtain the
  final convergence result, here, the restrictions for $\delta_n$ and $A_n$ are
  very loose: essentially boundedness is required. 
\end{itemize}
\end{REM}

As mentioned before, we want to take advantages out of $g$ being semi-convex.
The next lemmas are essential for that.
\begin{LEM} \label{lem:semi-convexity-g}
  Let $g$ be proper semi-convex with modulus $m\in\R$ with respect to the
  metric induced by $A\in\spd(\dimN)$. Then, for any $\bar x \in \dom \partial g$ 
  it holds that 
  \[
    g(x)  \geq g(\bar x) + \scal{\bar v}{x-\bar x} 
          + \frac{m}{2} \norm[A]{x-\bar x}^2 \,,
          \quad \forall x\in\R^\dimN\ \text{and}\ \bar v\in \partial g(\bar x) \,. 
  \]
\end{LEM}
\begin{proof}
Fix $\tilde x \in \dom g$ and apply the subgradient inequality to $g_m(x):= g(x) - \frac m2\norm[A]{x-\tilde x}^2$ around the point $\bar x$, i.e., it holds that
\[
  g_m(x) \geq g_m(\bar x) + \scal{\bar w}{x-\bar x}\,,\quad \forall x\in\R^\dimN \text{and}\ \bar w\in \partial g_m (\bar x) \,.
\]
Note that $\bar w$ is an element from the (convex) subdifferential. Due to the smoothness of $\frac{m}2 \norm[A]{x - \tilde x}^2$, we can use the summation rule for the limiting subdifferential to obtain 
\[
  \partial g_m(\bar x) =  \partial \left(g - \frac{m}{2}\norm[2]{\cdot - \tilde x}^2\right)(\bar x) = \partial g(\bar x) - mA(\bar x-\tilde x)\,,
\]
and, therefore, replacing $\bar w$ by $\bar v- mA(\bar x - \tilde x)$ with $\bar v \in \partial g(\bar x)$ in the subgradient inequality above, we obtain after using 
\[
  2\scal{\bar x-\tilde x}{x - \bar x}_A = \norm[A]{x - \tilde x}^2 - \norm[A]{\bar x-\tilde x}^2 - \norm[A]{x-\bar x}^2
\] 
that the following inequality holds
\[
  g_m(x) + \frac m2\norm[A]{x-\tilde x}^2 \geq g_m(\bar x) + \frac m2\norm[A]{\bar x-\tilde x}^2 + \frac{m}{2} \norm[A]{x-\bar x}^2 + \scal{\bar v}{x-\bar x}\,,\quad \forall x\in\R^\dimN\ \text{and}\ \bar v \in \partial g(\bar x)\,, 
\]
which implies the statement.
\end{proof}
\begin{LEM}\label{lem:eq:ipiano:descent-G}
  Let $\sigma = 1$ if $g$ is proper semi-convex with modulus $m\in\R$ with
  respect to the metric induced by $A\in\spd(\dimN)$ and $\sigma =0$ otherwise.
  Then it hold that
  \begin{equation} \label{eq:ipiano:descent-G}
    Q^\n(x^{n+1};x^\n) + \frac{\sigma}{2}\left(m + \frac 1{\alpha_n}\right) 
                            \norm[A]{x^{n+1}-x^n}^2 \leq Q^\n(x^n;x^\n) \,.
  \end{equation}
\end{LEM}
\begin{proof}
Apply Lemma~\ref{lem:semi-convexity-g} with $x=x^\n$ and $\bar x= x^\np$ to the
function $x\mapsto Q^\n(x;x^\n)$ from \eqref{eq:ipiano-vm-gen-up}, which is
semi-convex with modulus $\sigma\,(m+\frac{1}{\alpha_n})$ with respect to the
metric induced by $A$.
\end{proof}
\bigskip

\paragraph{Verification of Assumption~\ref{ass:Hs}.} 

We define the proper lower semi-continuous function 
\begin{equation} \label{eq:vm-ipiano-lyapunov-fun}
  \map{\F}{\R^\dimN\times\R^\dimN\times \R^{\dimN\times\dimN} \times \R}{\eR}
\quad\text{given by}\quad
  \F(x,y,A,\delta) := H_{(\delta,A)}(x,y)  := h(x) + \delta \norm[A]{x - y}^2
\end{equation}
for some $A\in\spd(\dimN)$ and $\delta \in \R$. Regarding the variables in
Assumption~\ref{ass:Hs}, the $u$-component of $\F$ is treated as
$u=(A,\delta)$, which allows the function $\F$ to change depending on the
metric $A$ and another parameter $\delta$. Convergence will be derived for the
$x$ and $y$ variables only.\\

The following proposition verifies \ref{ass:Hs:descent}, with $d_n =
\norm[2]{x^n-x^{n-1}}$ and $a_n = \gamma_n$.
\begin{PROP}[Descent property] \label{prop:vm-ipiano-Lyapunov-descent}
    Let the variables and parameters be given as in
    Algorithm~\ref{alg:ipiano-vm-gen}. Then, it holds that
    \begin{equation}\label{eq:hd-descent}
      H_{(\delta_n^\sigma,A_{n})}(x^{n+1},x^n) 
        \leq H_{(\delta_n^\sigma,A_n)} (x^{n},x^{n-1})
          -\gamma_n \varsigma (A_n) \norm[2]{x^n-x^{n-1}}^2\,,
    \end{equation}
    and the sequence $\seq[n\in\N]{H_{(\delta_n^\sigma,A_n)} (x^{n},x^{n-1})}$ 
    is monotonically decreasing, which verifies Condition~\ref{ass:Hs:descent} 
    with $\F$ as in \eqref{eq:vm-ipiano-lyapunov-fun}, 
    $d\pit\n=\norm[2]{x\iter\n-x\iter\nm}$, and
    $a_n = \gamma_n \varsigma (A_n)$.
\end{PROP}
\begin{proof} 
  Combining \eqref{eq:ipiano-vm-gen-up} (in the equivalent form \eqref{eq:ipiano-gen-up-variant-a}) with \eqref{eq:ipiano:quad-upper-bnd} and \eqref{eq:ipiano:descent-G} yields
  \begin{multline*}
    f(x^{n+1}) + g(x^{n+1}) + \frac{\sigma}{2}\left(m + \frac 1{\alpha_n}\right) \norm[A_n]{x^{n+1}-x^n}^2 \\
    \begin{split}
    \leq&\ f(x^n) + \scal{\nabla f(x^n)}{x^{n+1} - x^{n}} + \frac {L_n}{2} \norm[A_n]{x^{n+1}-x^n}^2 \\
    &\ + g(x^n) -\scal{\nabla f(x^n)}{x^{n+1} - x^{n}}+\frac{\beta_n}{\alpha_n}\scal{x^\np - x^\n}{x^\n-x^\nm}_{A_n}  - \frac{1}{2\alpha_n} \norm[A_n]{x^{n+1}-x^n}^2 \\
    = &\  f(x^n) + g(x^n) +\frac{\beta_n}{\alpha_n}\scal{x^\np - x^\n}{x^\n-x^\nm}_{A_n} + \Big(\frac {L_n}{2} - \frac 1{2\alpha_n}\Big) \norm[A_n]{x^{n+1}-x^n}^2 \\
    \end{split}
  \end{multline*}
  %Noting that $0\prec A_n \preceq \opid$ and using 
  %\[
  %  \norm[2]{x^{n+1}-x^n}^2 = \norm[A_n + (\opid-A_n)]{x^\np - x^\n}^2 \geq \norm[A_n]{x^\np - x^\n}^2
  %\] 
  and using $\scal{a}{b}_M \leq \frac 12 (\norm[M]{a}^2 + \norm[M]{b}^2)$ for any $a,b\in\R^\dimN$ and $M\in\spd(\dimN)$ implies the following inequality
  \[
    h(x^{n+1})  \leq h(x^n) + \frac{\beta_n}{2\alpha_n} \norm[A_n]{x^\n - x^\nm}^2
    - \frac 12\Bigg(\frac{1+\sigma - \beta_n}{\alpha_n} - (L_n-\sigma m) \Bigg) \norm[A_n]{x^\np - x^n}^2 
  \]
  and rearranging terms yields
  \[
    h(x^{n+1}) + \delta_n^\sigma \norm[A_n]{x^\np - x^n}^2   \leq h(x^n) + \delta_n^\sigma \norm[A_n]{x^\n - x^\nm}^2 - (\delta_n^\sigma - \frac{\beta_n}{2\alpha_n}) \norm[A_n]{x^\n - x^\nm}^2 \,.
  \]
\end{proof}
The parametrization of the step sizes is chosen as in \cite{Ochs15} (see
\cite[Lemma 6.3]{Ochs15} for well-definedness of the parameters.) Therefore, we
obtain the same step size restrictions here, but with the flexibility to change
the metric in each iteration.
\begin{REM}
  The proof shows that instead of \eqref{eq:ipiano-gen-up-variant-a} we could
  also consider
  \begin{equation} \label{eq:ipiano-gen-up-variant-b}
    \arg\min_{x\in\R^N}\ g(x) 
    + \scal{\nabla f(x^\n)}{x-x^\n}
    -\frac{\beta_n}{\alpha_n}\scal{x^\n-x^\nm}{x - x^\n} 
    + \frac{1}{2\alpha_n} \norm[A_n]{x - x^\n}^2  \,,
  \end{equation}
  which yields a slightly different algorithm, but step size restrictions are 
  the same. This expression differs from \eqref{eq:ipiano-gen-up-variant-a} in
  the metric of the inner product with coefficient ${\beta_n}/{\alpha_n}$.
\end{REM}

Next, we prove the relative error condition (Assumption~\ref{ass:Hs:error}) 
with $b_n \equiv 1$ and $\eps_n\equiv 0$, $I=\set{1,2}$, and 
$\theta_{1}=\theta_{2}=\frac 12$. First, we derive a bound on the (limiting)
subgradient of the function $h$ and then for the function $\F$.
\begin{LEM} \label{lem:vm-ipiano-error-bound}
  Let the variables and parameters be given as in Algorithm~\ref{alg:ipiano-vm-gen}. Then, there exists $b>0$ such that 
   \[
      \norm[-]{\partial h(x^\np)} \leq \frac b2 \left(\norm[2]{x^\np - x^\n} + \norm[2]{x^\n - x^\nm} \right) \,.
   \]
\end{LEM}
\begin{proof}
  \eqref{eq:ipiano-vm-gen-up-prox} can be used to specify an element from $\partial g(x^\np)$, namely
  \[
    A_n\frac{x^\n-x^\np}{\alpha_n} - \nabla f(x^\n) + \frac{\beta_n}{\alpha_n}A_n(x^\n - x^\nm) \in \partial g(x^\np)\,,
  \]
  which implies
  \[
    \norm[-]{\partial h(x^\np)} = \norm[-]{\nabla f(x^\np) + \partial g(x^\np)} \leq \left( \frac{\norm{A_n}}{\alpha_n} + L\right)\norm[2]{x^\np - x^\n} + \frac{\beta_n}{\alpha_n} \norm{A_n} \norm[2]{x^\n - x^\nm} \,.
  \]
  Using the Lipschitz continuity of $\nabla f$ and $A\preceq \opid$, the statement is verified.
\end{proof}
\begin{PROP} \label{prop:vm-ipiano-error-bound} 
  Let the variables and parameters be given as in 
  Algorithm~\ref{alg:ipiano-vm-gen}. Then, there exists $b>0$ such that 
  \[
    \norm[-]{\partial \F(x^\np, x^\n, A_\np, \delta^\sigma_\np)} 
    \leq \frac b2 \left(\norm[2]{x^\np - x^\n} + \norm[2]{x^\n - x^\nm} \right) \,,
  \]
  which verifies Condition~\ref{ass:Hs:error} with 
  $\F$ as in \eqref{eq:vm-ipiano-lyapunov-fun}, 
  $d\pit\n=\norm[2]{x\iter\n-x\iter\nm}$, 
  $b\pit\n\equiv1$, $I=\set{1,2}$, 
  $\theta_1=\theta_2=\frac 12$, and $\eps\pit\n\equiv0$.
\end{PROP}
\begin{proof}
  Thanks to summation rule of the limiting subdifferential for the sum of 
  $(x,y,A,\delta)\mapsto h(x)$ and the smooth function 
  $(x,y,A,\delta)\mapsto \delta\norm[A]{x^\np - x^\n}^2$, 
  we can compute the limiting subdifferential by estimating the partial 
  derivatives. We obtain
  \begin{gather} \label{eq:vm-ipiano-partial-subdiff}
    \partial_x \F (x,y,A,\delta) = \partial h(x) + 2\delta A(x-y)\,,\qquad 
    \partial_y \F (x,y,A,\delta) = \nabla_y \F (x,y,A,\delta) = -2A\delta A(x-y) \\
    \partial_A \F (x,y,A,\delta) = \nabla_A \F (x,y,A,\delta) = \delta(x-y)\tensor (x-y)\,, \qquad 
    \partial_\delta \F (x,y,A,\delta) = \nabla_\delta \F (x,y,A,\delta) = \norm[A]{x - y}^2 \,.
  \end{gather}
  In order to verify \ref{ass:Hs:error}, let 
  $\F^\np := \F(x^\np, x^\n, A_\np, \delta^\sigma_\np)$ 
  and we use $\norm[2]{w^\np} \leq \norm[2]{w_x^\np} + \norm[2]{w_y^\np} 
                                 + \norm[2]{w_A^\np} + \norm[2]{w_\delta^\np}$ 
  where $w\iter\np\in \partial \F^\np$ with block coordinates 
  $w_x^\np\in \partial_x \F^\np$, 
  $w_y^\np=\nabla_y \F^\np$, 
  $w_A^\np=\nabla_A \F^\np$, and 
  $w_\delta^\np = \nabla_\delta \F^\np$. We obtain the relative error bound 
  \ref{ass:Hs:error} using Lemma~\ref{lem:vm-ipiano-error-bound}, 
  $A_\np\preceq \opid$, boundedness of $\delta_\np^\sigma$, and the fact that 
  for a sequence $r_n\to 0$ for some $n_0\in\N$ it holds that 
  $r_n^2 \leq r_n$ for all $n\geq n_0$. In detail, we use 
  \[
    \norm[2]{w_A^\np} 
    \leq \delta^\sigma_\np \sum_{i,j} \abs{x_i^\np-x_i^\n}\cdot\abs{x_j^\np - x_j^\n} 
    \leq c \sum_{i,j} \abs{x_j^\np - x_j^\n} 
    \leq c c^\prime \sum_i \norm[2]{x^\np - x^\n} 
    \leq c c^\prime c^{\prime\prime} \norm[2]{x^\np - x^\n} \,,
  \]
  where $c$ is the maximal (over the coordinates $i$) bound for the converging
  sequences $\abs{x_i^\np - x_i^\n}\to 0$ as $\n \to\infty$, the dimensionally
  dependent constant $c^{\prime}=\sqrt{\dimX}$ provides the norm equivalence of
  $\norm[1]{\cdot}$ and $\norm[2]{\cdot}$, and $c^{\prime\prime}=\dimX$
  simplifies the summation. 
\end{proof}

The next proposition shows that converging subsequences of the sequence generated
by Algorithm~\ref{alg:ipiano-vm-gen} always $\F$-converge to the limit point, 
i.e. $\omega(x\iter0,u\iter0)=\omega_{\F}(x\iter0,u\iter0)$ is automatically
satisfied, which implies \ref{ass:Hs:cont} when the algorithm generates a 
bounded sequence.
\begin{PROP}\label{prop:vm-ipiano-Hs-cont}
  Let the variables and parameters be given as in 
  Algorithm~\ref{alg:ipiano-vm-gen}. Then, any convergent subsequence 
  $\seq[j\in\N]{(x^{n_j+1},x^{n_j},A_{n_j},\delta^\sigma_{n_j})}$ 
  actually $\F$-converges to a point $(x^*, x^*, A_*, \delta_*^\sigma)$, 
  which verifies Condition~\ref{ass:Hs:cont} for a bounded sequence 
  $\seq[\n\in\N]{(x\iter\n,u\iter\n)}$ with 
  $\F$ as in \eqref{eq:vm-ipiano-lyapunov-fun}.
\end{PROP}
\begin{proof}
  Let $(x^{n_j+1},x^{n_j},A_{n_j},\delta^\sigma_{n_j})$ be a subsequence 
  converging to some $(x^*, x^*, A_*, \delta_*^\sigma)$. 
  
  The continuity statement follows ($Q^\n(x^\np;x^n) \leq Q^\n(x;x^\n)$ for all 
  $x\in\R^\dimX$ from \eqref{eq:ipiano-vm-gen-up}) from
  \begin{multline*}
     g(x^{n_j+1}) + \scal{\nabla f(x^{n_j})}{x^{n_j+1}-x^{n_j}} 
      + \frac{1}{2\alpha_{n_j}} \norm[A_{n_j}]{x^{n_j+1}-y^{n_j}}^2  \\
      \leq 
        g(x^*) + \scal{\nabla f(x^{n_j})}{x^*-x^{n_j}} 
          + \frac{1}{2\alpha_{n_j}} \norm[A_{n_j}]{x^*-y^{n_j}}^2 \,.
  \end{multline*}
  Due to Lemma~\ref{lem:simple-conv}\ref{lem:simple-conv:iterates-converge} 
  $\norm{x^{n_j+1} -x^{n_j}} \to 0$, hence $\norm{y^{n_j} - x^{n_j}} \to 0$, 
  which shows that $y^{n_j} \to x^*$, as $j\to \infty$. Moreover, since 
  $f$ is continuously differentiable, $\nabla f(x^{n_j})$ converges as 
  $j\to\infty$, hence it is bounded. Therefore considering the limit superior 
  of $j\to\infty$ of both sides of the inequality shows that 
  $\limsup_{j\to\infty} g(x^{n_j+1}) \leq g(x^*)$, which combined with 
  the lower semi-continuity of $g$ implies $\lim_{j\to\infty} g(x^{n_j+1}) 
  = g(x^*)$, and thus the statement follows, since $f$ is continuously 
  differentiabe.
\end{proof}

Using the results that we just derived, we can prove convergence of the
variable metric iPiano method (Algorithm~\ref{alg:ipiano-vm-gen}) to a critical
point. Unlike the abstract convergence theorems in \cite{ABS13,FGP14,OCBP14},
the finite length property is derived for the coordinates from a subspace only,
which allows for a lot of flexibility. Critical points are characterized in the
proof of Proposition~\ref{prop:vm-ipiano-error-bound} (see
\eqref{eq:vm-ipiano-partial-subdiff}), where zero in the partial
subdifferential (actually the partial derivative) with respect to $y$, $A$, or
$\delta$ implies $x=y$ without imposing conditions on the $\delta$- or
$A$-coordinate. Thus, we have
\[
      0\in \partial \F(x,y,A,\delta) \Leftrightarrow \Big( 0 \in \partial h(x) \times  0_y \times 0_A \times 0_\delta \ \text{and}\ x=y \Big) \Leftrightarrow \Big( 0 \in \partial h(x)\ \text{and}\ x=y\Big) \,,
\]
where we indicate the size of the zero variables by the respective coordinate variable. As a consequence, $0\in \F(x^*,x^*,\delta,A) \Leftrightarrow 0 \in \partial h(x^*)$. These considerations lead to the following convergence theorem.
\begin{THM}\label{thm:vm-ipiano-KL-theorem}
  Suppose $\F$ in \eqref{eq:vm-ipiano-lyapunov-fun}, \eqref{eq:problem-class} 
  is a proper lower semi-continuous \KL function that is bounded from below. 
  Let $\seq[n\in\N]{x^n}$ be generated by Algorithm~\ref{alg:ipiano-vm-gen} 
  and bounded with valid variables and parameters as in the description of this
  algorithm. Then, the sequence $\seq[n\in\N]{x^n}$ satisfies 
  \begin{equation} \label{eq:KL-theorem-finite-length-dist}
    \sum_{k=0}^{\infty} \norm[2]{x^{k+1}-x^{k}} < +\infty\,,
  \end{equation}
  and $\seq[\n\in\N]{x^\n}$ converges to a critical point of \eqref{eq:problem-class}.
\end{THM}
\begin{proof}
  Verify the condition in Assumption~\ref{ass:Hs} and apply 
  Theorem~\ref{thm:KL-theorem-descent}. Set $d_n = \norm[2]{x^\n-x^\nm}$, 
  $a_n = \gamma_n\varsigma (A_n)$, $b_n\equiv1$, $\eps_n\equiv0$, $I=\set{1,2}$,
  $\theta_1=\theta_2=\frac 12$ then \ref{ass:Hs:descent}, \ref{ass:Hs:error}, 
  and \ref{ass:Hs:cont} are proved in Propositions~\ref{prop:vm-ipiano-Lyapunov-descent}, 
  \ref{prop:vm-ipiano-error-bound}, and~\ref{prop:vm-ipiano-Hs-cont}, and 
  \ref{ass:Hs:distance}, \ref{ass:Hs:params} are immediate from the bounds on 
  the parameters.
\end{proof}
\begin{REM}\label{rem:ipiano-semi-algebraic-instead-KL}
  Thanks to \cite{BDL06,BDLS07} the KL property holds for proper lower
  semi-continuous functions that are definable in an $\o$-minimal structure,
  e.g., semi-algrabraic functions. Since $\o$-minimal structures are stable
  under various operations, $\F$ is a KL function if $h$ is definable in an
  $\o$-minimal structure. Therefore, Theorem~\ref{thm:vm-ipiano-KL-theorem} can
  be applied to, for instance, a proper lower semi-continuous semi-algebraic
  function $h$ in \eqref{eq:problem-class}.
\end{REM}

%%%%%%%%%%%%%%%%%
%% New Section %%
%%%%%%%%%%%%%%%%%
\section{Block coordinate variable metric iPiano} \label{sec:bcvm-iPiano}

We consider a structured nonsmooth, nonconvex optimization problem with a
proper lower semi-continuous extended valued function $\map{h}{\R^{\dimN}}{\eR}$, 
$N\geq 1$, that is bounded from below by some value $\underline h >-\infty$:
\begin{equation}\label{eq:problem-class-block}
\min_{\bx \in \R^N} \; h(\bx)\,,
  \quad h(\bx) := f(\bx_1, \bx_2,\ldots,\bx_\dimJ) + \sum_{i=1}^\dimJ g_i(\bx_i)\,,
\end{equation}
where the $\dimN$ dimensions are partitioned into $\dimJ$ blocks of (possibly
different dimensions) $(\dimN_1,\ldots,\dimN_\dimJ)$, i.e., $\bx\in\R^\dimN$
can be decomposed as $\bx=(\bx_1,\ldots,\bx_\dimJ)$. The function $\map
f{\R^N}{\R}$ is assumed to be block $\cd 1$-smooth (possibly nonconvex) with
block Lipschitz continuous gradient on $\dom g_1\times \dom g_2\times \ldots
\times \dom g_\dimJ$, i.e., $\bx_i \mapsto \nabla_{\bx_i} f(\bx_1,\ldots,
\bx_i,\ldots, \bx_\dimJ)$ is Lipschitz continuous. Further, let the function
$\map {g_i} {\R^{\dimN_i}}{\eR}$ be simple (possibly nonsmooth and nonconvex)
and prox-bounded. 

Working with block algorithms can be simplified by an appropriate notation,
which we introduce now. We denote by $\bx_{\ov i}:=
(\bx_1,\ldots,\bx_{i-1},\bx_{i+1},\ldots,\bx_\dimJ)$ the vector containing all
blocks but the $i$th one. 

Algorithm~\ref{alg:ipiano-bvm-gen} is a straightforward extension of 
Algorithm~\ref{alg:ipiano-vm-gen} to problems of class 
\eqref{eq:problem-class-block} with a block coordinate structure. In each 
iteration, the algorithm applies one iteration of iPiano to the problem 
restricted to a certain block. The formulation of the algorithm allows blocks 
to be updated in an almost arbitrary order. In the end, the only restriction 
is that each block must be updated infinitely often, which is a more flexible
rule than in \cite{PS16}.

We seek for a critical point $\bx^*\in \dom h$ of $h$, which in this case is characterized by
\[
  -\nabla f(\bx) \in \partial g_1(\bx_1) \times \partial g_2(\bx_2) \times \ldots \times \partial g_\dimJ(\bx_\dimJ) \,.
\]
\begin{figure}[h!]
\centering
\setlength{\fboxrule}{1pt}
\fbox{
\begin{minipage}{0.95\textwidth}
\begin{ALG}\label{alg:ipiano-bvm-gen}
\
\textbf{Block coordinate variable metric iPiano}
\begin{itemize}
\item \textbf{Parameter}: Let for all $i\in\set{1,\ldots,\dimJ}$ 
\begin{itemize}
  \item $\seq[n\in\N]{\alpha_{n,i}}$ be a sequence of positive step size parameters,
  \item $\seq[n\in\N]{\beta_{n,i}}$ be a sequence of non-negative parameters, and
  \item $\seq[\n\in\N]{A_{\n,i}}$ be a sequence of matrices $A_{\n,i}\!\in\spd(\dimN_i)$ such that $A_{n,i}\! \preceq \opid$ and {$\inf_{n,i} \varsigma (A_{n,i}) > 0$}.
  \item Let $\sigma_i = 1$ if $g_i$ is semi-convex and $\sigma_i= 0$ otherwise.
\end{itemize}
\item \textbf{Initialization}: Choose a starting point $x^0 \in \dom h$ and set $x^{-1}= x^0$.
\item \textbf{Iterations $(n\ge 0)$}:  Update: Select $j_n \in \set{1,\ldots, \dimJ}$ and compute
  \begin{equation}\label{eq:ipiano-bvm-gen-up}
    \begin{split}
    \by^\n_{j_n} =&\ \bx^\n_{j_n} + \beta_{n,j_n}(\bx^\n_{j_\n} - \bx^\nm_{j_\n})  \\
    \bx^\np_{j_n} \in&\ {\arg\min_{x\in\R^{N_{j_n}}}}\ Q^{j_\n} (x;\bx^\n_{j_n}) \\[-2mm]
        &\ {\color{white}\arg\min_{x\in\R^{N_{j_n}}}}\  Q^{j_\n} (x;\bx^\n_{j_n}) := g_{j_n}(x) + \scal{\nabla_{\bx_{j_n}} f(\bx^\n)}{x-\bx^\n_{j_n}} + \frac{1}{2\alpha_{n,j_n}} \norm[A_{n,j_n}]{x - \by^\n_{j_n}}^2   \\[-1mm]
    \bx^\np_{\ov j_n} =&\ \bx^\n_{\ov j_n} \\
    \bx^\n_{\ov j_n} =&\ \bx^\nm_{\ov j_n} \,,
    \end{split}
  \end{equation}
  where $L_n>\sigma m_n$ is determined such that 
  \begin{equation} \label{eq:bipiano:quad-upper-bnd}
    f(\bx^{n+1}) \leq f(\bx^n) + \scal{\nabla_{\bx_{j_n}} f(\bx^n)}{\bx_{j_n}^\np - \bx_{j_n}^{n}} + \frac {L_n}{2} \norm[A_{n,j_n}]{\bx_{j_n}^\np-\bx_{j_n}^n}^2
  \end{equation}
  holds and $\alpha_{n,j_n},\beta_{n,j_n}$ with $\inf_{n,j} \alpha_{n,j}>0$ are chosen such that 
    \begin{equation} \label{eq:bipiano-delta-gamma}
      \delta_{n,j_n}^{\sigma_{j_n}}:= \frac 12\left(\frac{1+\sigma_{j_n} - \beta_{n,j_n}}{\alpha_{n,j_n}} - (L_n-\sigma_{j_n} m_\n) \right)\quad\text{and}\quad
      \gamma_{n,j_n} := \delta_{n,j_n}^{\sigma_{j_n}} - \frac{\beta_{n,j_n}}{2\alpha_{n,j_n}} 
    \end{equation}
  satisfy 
  \begin{equation} \label{eq:bipiano-delta-gamma-cond}
    \inf_{\n,j} \gamma_{\n,j}>0\quad\text{and}\quad \delta_{\np,j_n}^{\sigma_{j_n}} \norm[A_{\np,j_n}]{\bx_{j_n}^\np - \bx_{j_n}^n}^2 \leq \delta_{n,j_n}^{\sigma_{j_n}} \norm[A_{n,j_n}]{\bx_{j_n}^\np - \bx_{j_n}^n}^2  \,,
  \end{equation}
  where $m_n\in\R$ denotes the semi-convexity modulus of $g_{j_n}$ w.r.t. $A_{j_\n}\in\spd(\dimN_{j_n})$ (if $\sigma_{j_n}=1$).\\
  Set $A_{\np,\ov j_n} = A_{\n,\ov j_n}$, $\delta^{\sigma_{j_n}}_{\np,\ov j_n} = \delta^{\sigma_{j_n}}_{\n,\ov j_n}$.
  \end{itemize}
\end{ALG}
\end{minipage}
}
\end{figure}
In fact if we apply Algorithm~\ref{alg:ipiano-bvm-gen} to
\eqref{eq:problem-class} from the preceding section (i.e. $\dimJ=1$), we
recover the variable metric iPiano algorithm
(Algorithm~\ref{alg:ipiano-vm-gen}). For $\beta_{n,i} = 0$ for all $n\in\N$ and
$i\in\set{1,\ldots,\dimJ}$, the algorithm is known as Block Coordinate Variable
Metric Forward-Backward (BC-VMFB) algorithm \cite{CPR16}. If, additionally
$A_{n,i} = \opid$ for all $n$ and $i$, the algorithm is referred to as Proximal
Alternating Linearized Minimization (PALM) \cite{BST14}. An inertial block 
coordinate version (without variable metric) is proposed in \cite{PS16} as iPALM.

\paragraph{Verification of Assumption~\ref{ass:Hs}.} 
In order to prove convergence of this algorithm, we can make use of the results
of the preceding section for the variable metric iPiano algorithm. We consider
a function
\begin{equation} \label{eq:bvm-ipiano-lyapunov}
  \map{\F}{\R^{\dimN}\times\R^\dimN \times \R^{\dimN_1\times\dimN_1}\times \ldots\times \R^{\dimN_\dimJ\times\dimN_\dimJ} \times \R^\dimJ}{\eR}
\end{equation}
given by (set $\bA := (A_1,\ldots, A_\dimJ)$, $A_i\in\R^{\dimN_i\times\dimN_i}$, $\bdelta := (\delta_1,\ldots, \delta_\dimJ)$)
\[
  \F(\bx, \by, \bA,\bdelta) = \bH_{\bdelta,\bA}(\bx,\by) := h(\bx) + \sum_{i=1}^\dimJ \delta_i \norm[A_i]{\bx_i - \by_i}^2 \,.
\]

\begin{THM}\label{thm:bvm-ipiano-KL-theorem}
  Suppose $\F$ in \eqref{eq:bvm-ipiano-lyapunov}, \eqref{eq:problem-class-block} 
  is a proper lower semi-continuous \KL function (e.g. $h$ is
  semi-algebraic; cf. Remark~\ref{rem:ipiano-semi-algebraic-instead-KL}) that
  is bounded from below. Let $\seq[n\in\N]{\bx^n}$ be generated by
  Algorithm~\ref{alg:ipiano-bvm-gen} and bounded with valid variables and 
  parameters as in the description of this algorithm. Assume that each block 
  coordinate is updated after a finite number of $n^\prime\in\N$ steps. 
  Then, the sequence $\seq[n\in\N]{\bx^n}$ satisfies 
  \begin{equation} \label{eq:KL-theorem-finite-length-dist}
    \sum_{k=0}^{\infty} \norm[2]{\bx^{k+1}-\bx^{k}} < +\infty\,,
  \end{equation}
  and $\seq[\n\in\N]{\bx^\n}$ converges to a critical point of \eqref{eq:problem-class-block}.
\end{THM}
\begin{proof}
  As the $n$th iteration of Algorithm~\ref{alg:ipiano-bvm-gen} reads exactly the same as in Algorithm~\ref{alg:ipiano-vm-gen} but applied to the block coordinate $j_n$ only, we can directly apply Propositions~\ref{prop:vm-ipiano-Lyapunov-descent}, and obtain 
  \[
      \bH_{(\bdelta_\n^\sigma, \bA_n)}(\bx^{n+1},\bx^n) \leq \bH_{(\delta_n^\sigma,A_n)} (\bx^{n},\bx^{n-1})
        -\gamma_{n,j_n}\varsigma (A_{n,j_n}) \norm[2]{\bx_{j_n}^n-\bx_{n_j}^{n-1}}^2\,,
  \]
  and the function $\bH$ is monotonically decreasing along the iterations, i.e., the parameters in the algorithm are chosen such that one step on an arbitrary block decreases the value of $\bH$ unless the block coordinate is already stationary.

  Since the non-smooth part of the optimization problem \eqref{eq:problem-class-block} is additively separated the estimation of the subdifferential is easy as it reduces to the Cartesian product of the subdifferential with respect to each block. Therefore, Proposition~\ref{prop:vm-ipiano-error-bound} can be used analogously to deduce
  \[
    \norm[-]{\partial \F(\bx^\np, \by^\np, \bA_\np,\bdelta_\np)} \leq \frac b2 \left(\norm[2]{\bx^\np_{j_n} - \bx^\n_{j_n}}^2 + \norm[2]{\bx^\n_{j_n} - \bx^\nm_{j_n}}^2  \right) \,.
  \]

  Under the assumption that each block is updated at least after $n^\prime$ 
  iterations, also the continuity results from 
  Proposition~\ref{prop:vm-ipiano-Hs-cont} can be transferred easily to the 
  setting of Algorithm~\ref{alg:ipiano-bvm-gen},
  i.e., we can conclude that any convergent subsequence of block coordinates
  actually $\F$-converges to the limit point ($\lim_{k\to \infty}
  g_i(\bx_i^{n_k}) = g_i(\bx^*_i)$ for each block $i\in\set{1,\ldots,\dimJ}$
  and $f$ is continuous anyway).

  Therefore, the conditions in Assumption~\ref{ass:Hs} are verified by 
  $d_n = \norm[2]{\bx^\n_{j_n} - \bx^\nm_{j_n}}$, 
  $a_n=\gamma_{n,j_n}\varsigma (A_{n,j_n}) $, 
  $u_n=(\bdelta_\n^\sigma, \bA_n)$,
  $b_n\equiv1$, $\eps_n\equiv0$,
  $I=\set{1,2}$, and $\theta_1=\theta_2=\frac 12$. \ref{ass:Hs:distance} is 
  also satisfied because of the finite repetition of the updates, and 
  \ref{ass:Hs:params} is clearly satisfied.
\end{proof}

%%%%%%%%%%%%%%%%
%% New Section %%
%%%%%%%%%%%%%%%%
\section{Numerical application} \label{sec:num}

\subsection{A Mumford--Shah-like problem}

The continuous Mumford--Shah problem is given formally by
\begin{equation} \label{eq:MS}
  \min_{\img,\Gamma}\, \frac \lambda2 \INT[\Omega]{ \abs{\img-\noisy}^2 }x + \INT[\Omega\smallsetminus \Gamma] {\abs{\nabla \img}^2}x + \gamma \abs{\Gamma} \,,
\end{equation}
where $\map \img{\Omega}{\R}$ is an image on the image domain $\Omega \subset \R^2$ and $\map{\noisy}{\Omega}{\R}$ is a given noisy image, $\abs{\Gamma}$ measures the length of the jump set $\Gamma$. Intuitively, a solution $\img$ must be smooth except on a possible jump set $\Gamma$, and approximate $\noisy$. The positive parameters $\lambda$ and $\gamma$ steer the importance of each term. In order to solve the problem, the jump set $\Gamma$ needs to represented with a mathematical object that is amenable for a numerical implementation. 

Therefore, we consider the well-known Ambrosio--Tortorelli approximation \cite{AT90} given by
\begin{equation} \label{eq:MS-AT}
  \min_{\img,\edge}\, \frac \lambda2 \INT[\Omega]{ \abs{\img-\noisy}^2 }x  + \INT[\Omega] {\edge^2\abs{\nabla \img}^2} x + \gamma \INT[\Omega] {\eps \abs{\nabla \edge}^2 + \frac{(\edge-1)^2}{4\eps}}x \,,
\end{equation}
where $\eps >0$ is a fixed parameter and $\map \edge \Omega [0,1]$ is a (soft) edge indicator function, also called a phase-field. The last integral is shown to Gamma-converge to the length of the jump set of \eqref{eq:MS} as $\eps\to 0$. 

\begin{figure}
\begin{center}
\subfigure[\label{fig:monroe-inpaint-orig}original image $\noisy$]%
          {\includegraphics[width=0.24\linewidth]{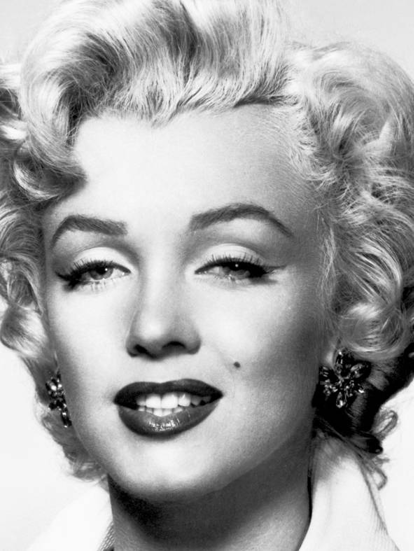}}
\subfigure[\label{fig:monroe-inpaint-mask}mask $\mask$ ($90\%$ unknown)]
          {\includegraphics[width=0.24\linewidth]{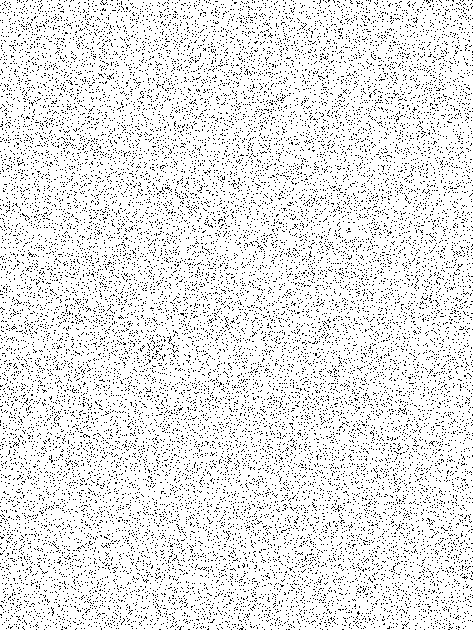}}
\subfigure[\label{fig:monroe-inpaint-AT}inpainting $\img$ using \eqref{eq:MS-AT-inpaint}]%
          {\includegraphics[width=0.24\linewidth]{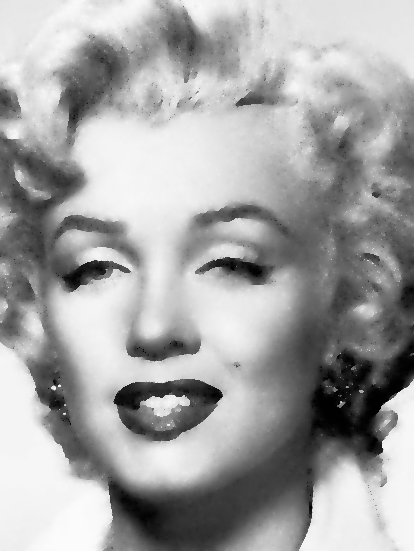}}
\subfigure[\label{fig:monroe-inpaint-LinDiff}linear diffusion inpainting]%
          {\includegraphics[width=0.24\linewidth]{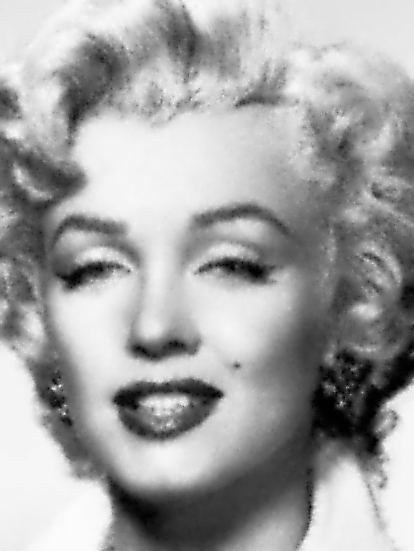}}
\end{center}
\caption{\label{fig:monroe-inpaint}Example for image inpainting/compression. The gray values of the original image (a) are stored only at the mask points (b), where known values are black $[\mask=1]$ and unknown ones are white $[\mask = 0]$. Based on $10\%$ known gray values the original image is reconstructed in (c) with the Ambrosio--Tortorelli inpainting \eqref{eq:MS-AT-inpaint} that we evaluate algorithmically in this paper, and in (d) with a simple linear diffusion model \cite{MW09} which arises as a special case of \eqref{eq:MS-AT-inpaint} when the edge set $\edge$ is fixed to $1$ everywhere on the image domain $\Omega$.}
\end{figure}
In this section, we solve a slight variation of this problem. Instead of an image denoising model we are interested in an inpainting problem (as shown in Figure~\ref{fig:monroe-inpaint}), which is usually more difficult. In image inpainting, the true information about the original image is only given on a subset $[\mask = 1]$ of the image domain (black pixels in Figure~\ref{fig:monroe-inpaint-mask}), where $\map\mask{\Omega}{\set{0,1}}$---the original image $\noisy$ is unknown on $[\mask =0]$ (white part Figure~\ref{fig:monroe-inpaint-mask}). In \cite{GWWBBS05}, the idea of image inpainting is pushed to a limit and used for PDE-based image compression, i.e., the inpainting mask $[\mask=1]$ is a small subset of $\Omega$. Usually a simple PDE is used for reconstructing the original image based on its gray values given only on mask points, for instance linear diffusion in \cite{MW09} (result given in Figure~\ref{fig:monroe-inpaint-LinDiff}). When the inpainting mask is optimized, linear diffusion based inpainting is shown to be competetive with JPEG and sometimes with JPEG2000. Therefore using a more general inpainting model combined with an optimized inpainting mask is expected to improve this performance. We consider the model
\begin{equation} \label{eq:MS-AT-inpaint}
  \begin{split}
  \min_{\img,\edge}&\ \INT[\Omega] {\edge^2\abs{\nabla \img}^2} x + \gamma \INT[\Omega] {\eps \abs{\nabla \edge}^2 + \frac{(\edge-1)^2}{4\eps}}x \\
  \st&\ \img(x) = \noisy(x)\,,\quad  \forall x\in [\mask = 1] \,,
  \end{split}
\end{equation}
which extends the linear diffusion model by optimizing for an additional edge set $\edge$. The linear diffusion model is recovered when fixing $\edge=1$ on $\Omega$. Since we want to evaluate our algorithms, we neglect the development made for finding an optimal inpainting mask and generate the mask by randomly selecting $10\%$ as known pixels. 

From now on, we discretize the problem and with a slight abuse of notation. We use the same symbols to denote the discrete counterparts of the above introduced variables: $\noisy\in\R^\dimN$ is the (vectorized\footnote{The columns of the image are stacked to a long vector.}) original image, $\mask\in\R^\dimN$ is the (inpainting) mask, $\img\in\R^\dimN$ is the optimization variable (representing a vectorized image), and $\edge\in [0,1]^\dimN$ represents the jump (or edge) set of \eqref{eq:MS}. The continuous gradient $\nabla$ is replaced by a discrete derivative operator $D\in\R^{2\dimN\times\dimN}$ that implements forward differences in horizontal $D_1\in\R^{\dimN\times\dimN}$ and vertical direction $D_2\in\R^{\dimN\times\dimN}$ with homogeneous boundary conditions, i.e., forward differences across the image boundary are set to $0$. Our discretized model of \eqref{eq:MS-AT-inpaint} reads
\begin{equation}  \label{eq:MS-AT-inpaint-discrete}
  \begin{split}
  \min_{\img,\edge}&\ \frac 12\norm[2]{\diag(\edge) (D_1 \img) }^2 +\frac 12\norm[2]{\diag(\edge) (D_2 \img) }^2 + \frac {\gamma\eps}2 \norm[2]{D \edge}^2 + \frac{\gamma}{4\eps}\norm[2]{\edge-1}^2 \\
  \st&\ \img_i = \noisy_i\,,\quad \forall i\in\set{1,\ldots,\dimN}\ \text{with}\ \mask_i = 1 \,,
  \end{split}
\end{equation}
where $\map{\diag}{\R^\dimN}{\R^{\dimN\times\dimN}}$ puts a vector on the diagonal of a matrix. Figure~\ref{fig:monroe-inpaint-sol} shows the input data, the reconstructed image, and the reconstructed edge set, for $\eps=0.1$ and $\gamma=1/400$ and the number of pixel $\dimN = 551\cdot 414 =228114$.

In the following, we evaluate several algorithms that use a variable metric. Let 
\begin{gather*}
  g_1(\img) := \ind X(\img)\text{ with } X:=\set{\img\in\R^\dimN\vert\, \img_i = \noisy_i\ \text{if}\ \mask_i = 1 }\,,\qquad
  g_2(\edge) := \frac{\gamma}{4\eps}\norm[2]{\edge-1}^2 \\
  f(\img,\edge) := \frac 12 \left( \norm[2]{\diag(\edge) (D_1 \img) }^2 +\norm[2]{\diag(\edge) (D_2 \img) }^2 + \gamma\eps \norm[2]{D \edge}^2 \right)\,.
\end{gather*}
We can apply iPiano to \eqref{eq:problem-class} with $x=(\img,\edge)$ and $g(x) = (g_1(\img),g_2(\edge))$, or block coordinate iPiano to \eqref{eq:problem-class-block} with $\bx_1=\img$ and $\bx_2=\edge$.

In order to determine a suitable metric, we first compute the derivatives of $f$
\[
  \begin{split}
  \nabla_\img f(\img,\edge) =&\  \left( D_1^\top \diag(\edge^2) D_1+ D_2^\top \diag(\edge^2) D_2 \right)\img  \\
  \nabla_\edge f(\img,\edge) =&\ \left( \diag( (D_1\img)^2 ) + \diag( (D_2\img)^2) + \gamma\eps D^\top D \right) \edge \,,
  \end{split}
\]
where the squares are to be understood coordinate-wise. A feasible metric for block coordinate variable metric iPiano (\BCVMiPiano) must satisfy \eqref{eq:bipiano:quad-upper-bnd}. Therefore, for the $\img$-update step ($\edge$ is fixed), we require $A_{n,\img}$ (the metric w.r.t. the block of $\img$ coordinates) to satisfy
\[
  \scal{\nabla_\img f(\img,\edge) - \nabla_\img f(\img^\prime,\edge) - A_{n,\img}(\img - \img^\prime)}{\img - \img^\prime} \leq 0 
\]
for all $\img, \img^\prime$, which is achieved, for example, by a diagonal matrix $A_{n,\img}$ given by
\begin{equation} \label{eq:AT-metric-img}
  (A_{n,\img})_{i,i} = \sum_{j=1}^\dimN \abs{\left(  D_1^\top \diag(\edge^2) D_1+ D_2^\top \diag(\edge^2) D_2 \right)_{i,j}}
\end{equation}
for all $i\in\set{1,\ldots, \dimN}$. In order to avoid numerical problems, we 
add a small numerical constant $10^{-9}$ to the diagonal of $A_{n,\img}$. 
For the $\edge$-update ($\img$ is fixed), analogously, we require $A_{n,\edge}$ (the metric w.r.t. the block of $\edge$ coordinates) to satisfy
\[
  \scal{\nabla_\img f(\img,\edge) - \nabla_\img f(\img,\edge^\prime) - A_{n,\edge}(\edge - \edge^\prime)}{\edge - \edge^\prime} \leq 0 
\]
for all $\edge, \edge^\prime$, which is achieved, for example, by a diagonal matrix $A_{n,\edge}$ given by
\begin{equation} \label{eq:AT-metric-edge}
  (A_{n,\edge})_{i,i} = \sum_{j=1}^\dimN \abs{\left(   \diag( (D_1\img)^2 ) + \diag( (D_2\img)^2) + \gamma\eps D^\top D  \right)_{i,j}} 
\end{equation}
for all $i\in\set{1,\ldots, \dimN}$. Note that compared to \eqref{eq:bipiano:quad-upper-bnd} the metric contains the scaling $L_{n,\img}$ and $L_{n,\edge}$, respectively. For constant step size schemes ($A_{n,\img}=A_{n,\edge}=\opid$) we use $L_\img \leq 8$ and\footnote{Note that $\noisy$ is normalized to $[0,1]$ and, thus, we observed that $\img$ stays in $[0,1]$ too. Therefore $(D_1\img)_i^2$ is in $[0,1]$.} $L_\edge \leq 2+8 \gamma\eps$. 

\begin{figure}[t]
\begin{center}
  \includegraphics[width=0.8\linewidth]{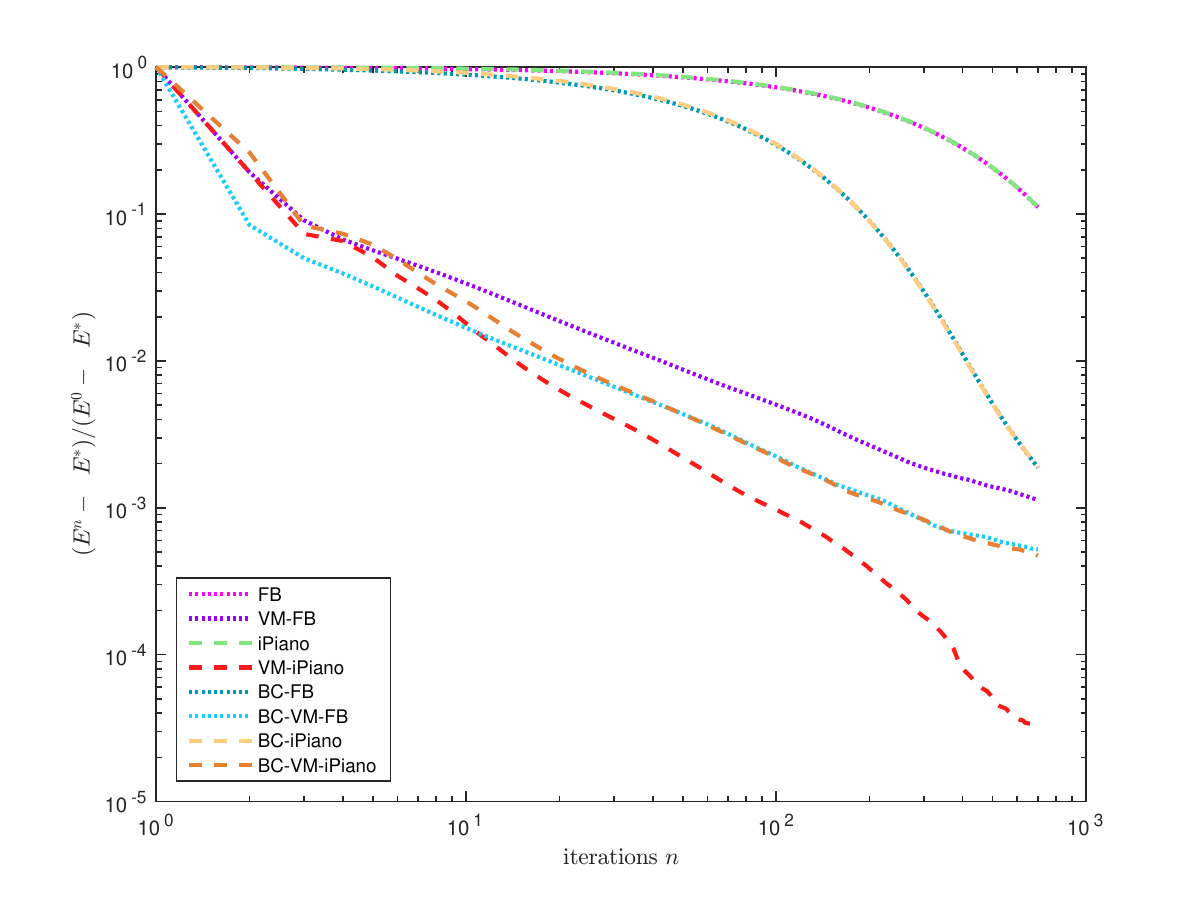}
\end{center}
\caption{\label{fig:AT-inpaint-eval}Number of iterations vs. relative objective
value for solving \eqref{eq:MS-AT-inpaint-discrete}. The performance is
significantly improved for methods that take a variable metric into account.
Intuitively, this means that the coordinates of the optimization variable are
irregularly scaled along the iterations. The variable metric version of iPiano
shows the best performance.}% Due to the scaling FB, BC-FB are not visible. The 
%performance plot of FB equals that of iPiano, }
\end{figure}
Besides \BCVMiPiano, we test forward--backward splitting (\FB) with constant step size scheme $\alpha=2/\max(L_\img,L_\edge)$, block coordinate forward--backward splitting (\BCFB) with step sizes $\alpha_\img = 2/L_\img$ and $\alpha_\edge = 2/L_\edge$ (this method is also known as PALM \cite{BST14}), variable metric forward--backward splitting (\VMFB) with the metric \eqref{eq:AT-metric-img} and \eqref{eq:AT-metric-edge} as a composed diagonal matrix, block coordinate variable metric forward--backward splitting (\BCVMFB) with the metric \eqref{eq:AT-metric-img} and \eqref{eq:AT-metric-edge}, iPiano (\iPiano) with constant step size scheme $\alpha=2(1-\beta)/\max(L_\img,L_\edge)$, block coordinate iPiano (\BCiPiano) with constant step size scheme $\alpha_\img = 2(1-\beta)/L_\img$ and $\alpha_\edge = 2(1-\beta)/L_\edge$, variable metric iPiano (\VMiPiano) with the metric \eqref{eq:AT-metric-img} and \eqref{eq:AT-metric-edge} as a composed diagonal matrix, and block coordinate variable metric iPiano (\BCVMiPiano) with the metric \eqref{eq:AT-metric-img} and \eqref{eq:AT-metric-edge}. For all methods that incorporate an inertial parameter, it is set to $\beta = 0.7$. 

The metric that is used for \VMFB and \VMiPiano is actually not feasible, as \eqref{eq:AT-metric-img} and \eqref{eq:AT-metric-edge} are not sufficient to guarantee that the metric induces a quadratic majorizer to the function $f$ (cf. \eqref{eq:ipiano:quad-upper-bnd}). The gradient is not linear with respect to both coordinates. The gradient is linear only if one coordinate is fixed. Nevertheless, in our practical experiments, the methods converged. In future work, we want to analyze if this inaccuracy can be compensated by making use of relative error conditions, which are not yet incorporated into the algorithms.

We solve problem \eqref{eq:MS-AT-inpaint-discrete} with all methods up to 1000 iterations and define $E^*$ as the minimal objective value  that is achieved among all methods. Let $E^0$ be the initial value. Figure~\ref{fig:AT-inpaint-eval} plots the decrease of the relative objective value $(E^\n-E^*)/(E^0-E^*)$ along the iterations $\n$ on a logarithmic scale on both axes.

The performance of $\FB$ and $\iPiano$ are nearly identical as they do not explore the different scaling of $\img$- and $\edge$-coordinates, unlike $\BCFB$ and $\BCiPiano$. As both block coordinates seem to \enquote{live} on a different scale, block coordinate methods are favorable. However, as the immense performance speed up of the variable metric methods shows the irregular scaling happens to be present also among different $\img$-coordinates, respectively, $\edge$-coordinates. Throughout the experiments, we have noticed that optimization problems where regularization (like smoothness between pixels) is important, inertial methods seem to perform slightly better in general. For this experiment variable metric iPiano shows the best performance and sets the value for $E^*$, the lowest objective value among all methods after 1000 iterations. Note that the computational cost per iterations is nearly the same for all methods.

%The quality of the reconstruction is measured by the peak-signal-to-noise-ratio (PSNR), which is defined by $\text{PSNR} = -10\cdot \log_{10}(\text{MSE})$, where MSE denotes the mean-squared error given by $\text{MSE}=\frac 1{\dimN}\sum_{i=1}^\dimN (\img_i - \noisy_i)^2$. The higher the PSNR value the better. The PSNR for the linear diffusion based reconstruction in this example is 26.535, for the reconstruction using \eqref{eq:MS-AT-inpaint-discrete}, \FB achieves 26.237, \VMFB 25.557, \iPiano 26.242, \VMiPiano 25.176, \BCFB 25.497, \BCVMFB 25.282, \BCiPiano 25.489, \BCVMiPiano 25.336. 

\begin{figure}[t]
\begin{center}
\subfigure[\label{fig:monroe-inpaint-sol-input}input to model \eqref{eq:MS-AT-inpaint-discrete}]%
          {\includegraphics[width=0.24\linewidth]{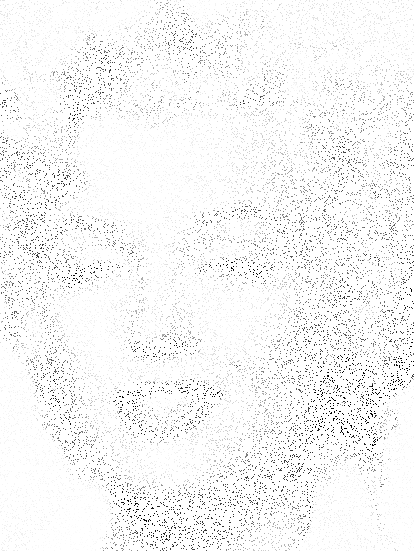}}
\subfigure[\label{fig:monroe-inpaint-sol-img}inpainting $\img$ using \eqref{eq:MS-AT-inpaint-discrete}]%
          {\includegraphics[width=0.24\linewidth]{monroe_inpaint.png}}
\subfigure[\label{fig:monroe-inpaint-sol-edges}edges $\edge$ using \eqref{eq:MS-AT-inpaint-discrete}]%
          {\includegraphics[width=0.24\linewidth]{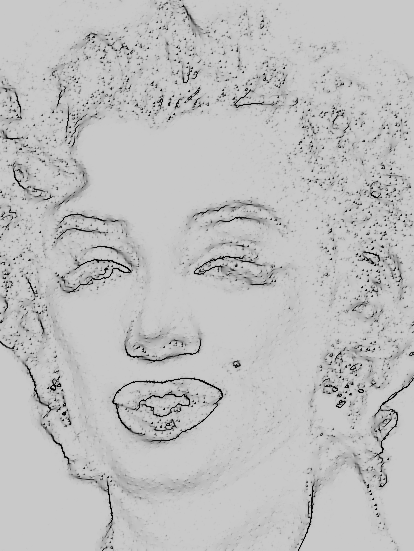}}
\end{center}
\caption{\label{fig:monroe-inpaint-sol}Solution to Problem~\ref{eq:MS-AT-inpaint-discrete}. (a) shows the inpainting mask from Figure~\ref{fig:monroe-inpaint-mask} weighted with the gray values from Figure~\ref{fig:monroe-inpaint-orig}. (b) shows the solution image $\img$ and (c) the solution edge set $\edge$ of \eqref{eq:MS-AT-inpaint-discrete}. Although the model is non-convex, visually all algorithms resulted in a similar solution. Figure~\ref{fig:AT-inpaint-eval} shows that the final objective values differ.}
\end{figure}

%%%%%%%%%%%%%%%%
%% CONCLUSION %%
%%%%%%%%%%%%%%%%
\section{Conclusion}

In this paper, we presented a convergence analysis for abstract inexact generalized descent
methods based on the KL-inequality that unifies and generalizes the analysis in
Attouch et al. \cite{ABS13}, Frankel et al. \cite{FGP14}, Ochs et al.
\cite{OCBP14}, Bolte and Pauwels \cite{BP16}, and several other more explicit 
algorithms. The novel convergence theorem allows for more flexibility in the
design of algorithms. More in detail, algorithms that imply a descent
on a proper lower semi-continuous parametric function and satisfy a certain
flexible relative error condition are considered. The parametric function can be seen as
an objective function that may vary along the iterations.
The gained flexibility is used to formulate a variable metric
version of iPiano (an inertial forward--backward splitting-like method).
Moreover, thanks to usage of a generic distance measure in the abstract
convergence theorem, we obtain a block coordinate variable metric version of
iPiano almost for free. Finally,
the algorithms are shown to perform well on the practical problem of image
compression using a Mumford--Shah-like regularization. 

As future work, we will investigate whether the gained flexibility can be used,
for example, to prove the convergence of (inertial) Bregman proximal descent
methods with Bregman functions that are not required to be strongly convex or
to have a Lipschitz continuous gradient.

%%%%%%%%%%%%%%
%% APPENDIX %%
%%%%%%%%%%%%%%
\appendix
\section{Appendix} 

\subsection{Relation to algorithms with analogue convergence guarantees}\label{sec:appdx}

In recent works, the convergence analysis of algorithms for non-smooth 
non-convex optimization problems often follows the lines of the proof 
methodology suggested in \cite{BST14}, i.e., the convergence is explicitly
verified, although it suffices to verify the abstract conditions in \cite{ABS13}.
In the following, for several such algorithms, the relation to the abstract 
conditions in \cite{ABS13,FGP14,OCBP14} and Assumption~\ref{ass:Hs} is shown. 
For \cite{LP16,LP15b,LFP16}, the generalizations of our paper are necessary to cast 
them into the abstract framework. Note that we do not provide an exhaustive list 
of examples. Most of the algorithms mentioned in the introduction fall into 
our unifying abstract setting.

\paragraph{Relation to PALM \cite{BST14}.} In \cite{BST14}, the general proof 
methodology is introduced. Thanks to a uniformization result of the KL-inequality, 
which we also use in this paper (see Lemma~\ref{lem:uniformization}), the 
convergence proof was simplified compared to \cite{ABS13}. 
\cite[Lemma 3(i)]{BST14} verifies \ref{ass:ABS13-Hs:descent}, 
\cite[Lemma 4]{BST14} shows \ref{ass:ABS13-Hs:error}, and 
\cite[Lemma 5(i)]{BST14} contains the continuity statement \ref{ass:ABS13-Hs:cont}.

\paragraph{Relation to \cite{BCL15}.} An inertial algorithm for the sum of 
two non-convex functions was proposed in this paper. The setting is slightly
more general than \cite{OCBP14} as the non-smooth part of the objective is
allowed to be non-convex. The proximal subproblems are formulated with respect
to Bregman distances that are required to be strongly convex and with Lipschitz
continuous gradient, which provides a lower and upper bound in the Euclidean
metric for the Bregman distance terms.  The proof of convergence is, hence,
analogue to \cite{OCBP14}. However, unlike in \cite{OCBP14}, the sufficient
decrease condition uses $d\pit\n=\norm[2]{x\iter\np-x\iter\n}$ instead of 
$\norm[2]{x\iter\n-x\iter\nm}$. Both conditions obviously fall into the more
general set of conditions in Assumption~\ref{ass:Hs}. The conditions in
Assumption~\ref{ass:Hs} are verified in \cite[(H1)--(H3) on page 13]{BCL15} in
analogy to \ref{ass:OCBP14-Hs:descent}--\ref{ass:OCBP14-Hs:cont} for which we
provide the details in Section~\ref{sec:rel-abstr-conv}.

\paragraph{Relation to \cite{LP16}.} A Douglas--Rachford splitting algorithm for 
solving non-smooth non-convex problems of the form
\begin{equation} \label{eq:LP16-min-prob}
  \min_{x\in \R^N}\, f(x) + g(x) \,,
\end{equation}
where $f$ has Lipschitz continuous gradient and $g$ is proper lower 
semi-continuous, is proposed. The algorithm generates sequences 
$\seq[\n\in\N]{x\iter\n}$, $\seq[\n\in\N]{y\iter\n}$, and $\seq[\n\in\N]{z\iter\n}$
according to the following update scheme: $(\gamma>0)$
\[
  \begin{split}
    y\iter\np \in&\ \argmin{y} f(y) 
                      + \frac 1{2\gamma} \norm[2]{y-x\iter\n}^2 \\
    z\iter\np \in&\ \argmin{z} g(z) 
                      + \frac 1{2\gamma} \norm[2]{2y\iter\np-x\iter\n -z }^2 \\
    x\iter\np =&\ x\iter\n + (z\iter\np - y\iter\np)
  \end{split}
\]
The global convergence of the whole sequence 
$\seq[\n\in\N]{(y\iter\n,z\iter\n,x\iter\n)}$ is shown in \cite[Theorem 2]{LP16}
for certain values of $\gamma>0$, and is based on a descent property of the 
merit function
\[
  \mathfrak D_\gamma(y,z,x) := f(y) + g(z) - \frac{1}{2\gamma} \norm[2]{y-z}^2 
                     + \scal{x-y}{z-y}\,.
\]
During the proof, which they tailored to their method, the abstract conditions 
in Assumption~\ref{ass:Hs} are verified. \ref{ass:Hs:descent} is verified 
in \cite[Eq. (23)]{LP16} with some constant $a>0$ for the function 
$\mathfrak D_\gamma$ using $d\iter\n:=\norm[2]{y\iter{\np}-y\iter{\n}}$
\[
  \mathfrak D_\gamma(y\iter\np,z\iter\np,x\iter\np) 
  + a\norm[2]{y\iter\np-y\iter\n}^2
  \leq \mathfrak D_\gamma(y\iter\n,z\iter\n,x\iter\n)\,,
\]
\ref{ass:Hs:error} is established in \cite[Eq. (28)]{LP16} for some $b>0$, 
\[
  \dist(0, \partial \mathfrak D_\gamma(y\iter\n,z\iter\n,x\iter\n))
  \leq b \norm[2]{y\iter\np - y\iter\n} \,,
\]
using $I:=\set{0}$, $\theta_0=1$, $b\pit\n\equiv1$, $\eps\pit\n\equiv0$, and 
\ref{ass:Hs:cont} is proved by assuming the existence of a cluster point
and the $\mathfrak D_\gamma$-attentive convergence from 
\cite[Eq. (25)--(27)]{LP16}. The distance condition \ref{ass:Hs:distance} is
asserted by \cite[Eq. (22),(10)]{LP16} and the relation in the $x$-update step.
\ref{ass:Hs:params} is obviously satisfied, since we are in a setting with
constant parameters. Therefore, we can apply our
Theorem~\ref{thm:KL-theorem-descent} to prove the same convergence results as in
\cite[Theorem 2]{LP16}: $\seq[\n\in\N]{y\iter\n}$ converges and, using the same 
equations that realize the distance condition, convergence of 
$\seq[\n\in\N]{z\iter\n}$ and $\seq[\n\in\N]{x\iter\n}$ can be concluded.

\paragraph{Relation to \cite{LP15b}.} In a similar way to \cite{LP16}, the 
proximal ADMM proposed in \cite{LP15b} can be cast into our framework. The goal
is to solve the following problem:
\[
  \min_{x\in \R^N}\, h(x) + P(\mathcal M x)\,,
\]
with a linear mapping $\mathcal M$, a proper lower semi-continuous function $P$, and a twice 
continuously differentiable function $h$ with bounded Hessian. The
sufficient decrease condition is proved for the Lagrange function 
\[
  L_\beta(x,y,z) = h(x) + P(y) - \scal{z}{\mathcal Mx-y} 
                  + \frac\beta2\norm[2]{\mathcal Mx-y}^2 
\]
in \cite[Eq. (36)]{LP15b} with $d\iter\n := \norm[2]{x\iter\np-x\iter\n}$, 
and some $a>0$. Different from the analysis in \cite{LP16}, where the relative
error condition is explicit, it is implicit in \cite{LP15b}. The condition
\ref{ass:Hs:error} is verified in \cite[Eq. (35)]{LP15b} for some $b>0$, 
$b\pit\n\equiv1$, $\eps\pit\n\equiv0$, $I=\set{1}$ and $\theta_1=1$. The 
condition \ref{ass:Hs:cont} is proved in \cite[Theorem 2(i)]{LP15b}. The 
distance condition \ref{ass:Hs:distance} follows directly from 
\cite[Eq. (14),(15)]{LP15b}, and \ref{ass:Hs:params} is again obviously 
satisfied.

\paragraph{Relation to \cite{LFP16}.} A very general multi-step forward--backward
scheme is proposed to solve problems of the setting of \eqref{eq:LP16-min-prob}.
The main update step is a forward--backward step, executed at an extrapolated 
point with gradient direction evaluated at another extrapolated point. Both of 
these extrapolations allow for a linear combination (possibly different ones) of 
finitely many preceding step directions. Global convergence and a finite length
property are proved in \cite[Theorem 2.2]{LFP16} explicitly for this algorithm 
for the sequence $\seq[\n\in\N]{x\iter\n}$ and $\seq[\n\in\N]{z\iter\n}$ with 
$z\iter\n=(x\iter\n,x\iter\nm,\ldots,x\iter{n-s+1})$ for some $s\in \N$. 
The statements that establishes the conditions in Assumption~\ref{ass:Hs} are 
collected in \cite[(R.1)--(R.3)]{LFP16} in the supplementary material. The proof
idea follows the concepts of the proof of iPiano \cite{OCBP14}. The 
arising Lyapunov function and the product space is naturally generalized to the 
number of terms used in the linear combinations of the extrapolations.

% ************************
% >>>>> bibliography <<<<<
% ************************
{\small
\bibliographystyle{ieee}
\bibliography{ochs}
}

\end{document}